\documentclass[12pt]{article}
\usepackage{graphicx,amsthm}
\usepackage[margin=1in]{geometry}
\usepackage[numbers]{natbib}

\usepackage{amsmath,amssymb}
\usepackage{tikz,bookmark,epsfig,rotating,colortbl}
\usepackage{multirow, anyfontsize}

\usepackage{amsthm}%
\usepackage{mathrsfs}%
\usepackage[title]{appendix}%
\usepackage{xcolor}%
\usepackage{textcomp}%
\usepackage{manyfoot}%
\usepackage{booktabs}%

\usepackage{listings,bm}
\usepackage{comment}
\usepackage{float}
\usepackage{hhline}
\usepackage{authblk}
\newtheorem{lemma}{Lemma}
\newtheorem{proposition}{Proposition}
\newtheorem{theorem}{Theorem}
\newtheorem{corollary}{Corollary}
\title{\textbf{On the Cores of Uniform and Almost-Uniform $3$-Qualitative Independence Hypergraphs}}
\author[1]{Raina Mary Thomas}
\author[2]{Yasmeen Akhtar}

\affil[1,2]{ Department of Mathematics, BITS Pilani K K Birla Goa Campus, Goa 403726, India}

\affil[1]{\texttt{p20220019@goa.bits-pilani.ac.in}}
\affil[2]{\texttt{yasmeena@goa.bits-pilani.ac.in}}

\date{}
\begin{document}

\maketitle

\definecolor{burgundy}{rgb}{0.5, 0.0, 0.13}
\definecolor{amethyst}{rgb}{0.6, 0.4, 0.8}
\maketitle

\begin{abstract}
  Qualitative independence hypergraphs provide a useful combinatorial framework for analyzing the existence and structure of covering arrays. In this work, we study the \emph{uniform} and \emph{almost-uniform $3$-qualitative independence hypergraphs} $3\text{-}UQI(n,2)$ and $3\text{-}AUQI(n,2)$, and establish a structural correspondence between these families and merged Johnson graphs, with emphasis on their cores. Focusing on the smallest unresolved instance, $3\text{-}QI(8,2)$, we classify all of its strongly independent sets and determine its strong independence number. Using this, along with its strong chromatic number and the size of the largest $3$-clique, we show that $3\text{-}QI(8,2)$ is a core. For $n>8$, we further identify sufficient conditions under which $3\text{-}UQI(n,2)$ and $3\text{-}AUQI(n,2)$ are cores.
\end{abstract}
\noindent\textbf{Keywords:} Hypergraph, Qualitative Independence, Merged Johnson Graph, Core, Strong Chromatic Number, Strong Independence Number\\
\noindent\textbf{MSC codes:} 05C65, 05C60, 05C35, 68R05, 05D05
\section{Introduction}\label{sec1}
 Let $S$ be a set of $g$ symbols, and $n$, $t$ be positive integers such that $t \geq 2$, $n \geq g^t$. A set of strings of length $n$ and entries from $S$ is said to be \emph{$t$-qualitatively independent} ($t$-QI) if for any $t$-subset, $\{v_{1}, v_{2}, \ldots , v_{t}\}$, of strings and any ordered $t$-tuple say $(g_1, g_2, \ldots, g_t)\in S^t$, there exists an index $j$ such that $v_{i}[j] = g_i$ for $i=1, 2, \ldots, t$. This interesting property, often studied in \emph{extremal set theory}, is what forms the basis of a special class of hypergraphs known as \emph{qualitative independence (QI) hypergraphs}. A \emph{hypergraph $H$} is a pair $H = (V, E)$, where $V = \{v_1, v_2, . . . , v_k\}$ is a set of vertices, and $E = \{e_1, e_2, . . . , e_m\}$ is a set of non-empty subsets of $V$, called \emph{hyperedges}, such that $\displaystyle \bigcup_{i=1} ^ {m} e_i = V$. A \emph{$t$-uniform} hypergraph is one in which all the hyperedges have the same cardinality $t$. A \emph{$t$-qualitative independence hypergraph}, $t\text{-}QI(n, g) = (V, E)$, is a $t$-uniform hypergraph, where $V$ is the set of all $g$-ary strings of length $n$ with entries from $S$ and with the property that every symbol appears at least $g^{t-1}$ times and the first appearance of every symbol is in the lexicographic order; a set of $t$ vertices forms a hyperedge if they are $t$-qualitatively independent. Equivalently, the vertices may be represented as partitions of an $n$-set into $g$ classes, such that the indices where a particular symbol occurs in the string belong to the same class in the partition. A variable strength version of these hypergraphs, where the hyperedges can be of any size and are closed under subsets, was introduced by Raaphorst et al. \cite{raaphorst2013thesisvariable, raaphorst2018variable} to characterize the existence of variable strength covering arrays. \\

 A \emph{covering array}, $t\text{-}CA(n, k, g)$, of strength $t$, $g$ symbols, $k$ factors, and size $n$, is an $n \times k$ array, where the entries are from a set $S$ of cardinality $g$ (say $\mathbb{Z}_g$), with the property that any set of $t$ column vectors is $t$-qualitatively independent. Therefore, covering arrays (CAs) are also known as \emph{qualitatively independent families}. They are used in combinatorial testing as a cost-effective and efficient tool to design test suites \cite{alanwilliams1996practical, survey2019methods, applstevens1998efficient, appl1994compressing}. Each column in a CA represents a factor that is tested, each entry is an input taken by the corresponding factor, and each row represents a test. An \emph{optimal} covering array is one with the least possible size. This parameter in practice, corresponds to the least number of tests needed. The size of an optimal covering array, $t\text{-}CAN(k, g)$, is known as the \emph{covering array number}. A \emph{covering array on a hypergraph} $H = (V, E)$, of size $n$ and $g$ symbols, denoted $CA(n, H, g)$, is an $n \times |V|$ array with entries from $S$ and columns indexed by the vertices in $V$ such that: for any hyperedge $e \in E$ with $|e| = m$, the columns in the sub-array corresponding to the vertices in $e$ are $m$-qualitatively independent. The \emph{covering array number of a hypergraph $H$} with $g$ symbols, is the size of an optimal CA on $H$ and is denoted as $CAN(H,g)$. An example of an optimal CA on a hypergraph $H=(V, E)$, where $V=\{a, b, c, d, e\}$ and $E=\{\{a,b,c\}, \{b, c, d\}, \{c,d,e\}, \{b,d,e\}\}$ is given in Figure \ref{fig:eg}.
\begin{center}
\begin{figure}
    \centering
    \begin{tikzpicture}[scale=4]
    \centering	
	\draw[rotate around={30:(0, 0.4)}, black, fill=burgundy!40, fill opacity=0.4] (0, 0.4) ellipse (0.15 cm and 0.55 cm);
	\draw[rotate around={150:(0.4, 0.4)}, black, fill=cyan!40, fill opacity=0.4] (0.4, 0.4) ellipse (0.15 cm and 0.55 cm);
	\draw [black, rounded corners=13mm, fill=green!40, fill opacity=0.4] (-0.3,0.5)--(0.7,0.5)--(0.2,-0.2)--cycle;
	  \draw [black, rounded corners=13mm, fill=black!40, fill opacity=0.4] (-0.35,0.3)--(0.65,0.3)--(0.75,1)--cycle;
        \node[draw, black, circle,thick,minimum size=0.1cm,inner sep=0pt, fill=black] at (0.2, 0.05) {};
        \node[draw, black, circle,thick,minimum size=0.1cm,inner sep=0pt, fill=black] at (0.4, 0.4) {};
        \node[draw, black, circle,thick,minimum size=0.1cm,inner sep=0pt, fill=black] at (0.6, 0.7) {};
        \node[draw, black, circle,thick,minimum size=0.1cm,inner sep=0pt, fill=black] at (0, 0.4) {};
        \node[draw, black, circle,thick,minimum size=0.1cm,inner sep=0pt, fill=black] at (-0.2, 0.7) {};
        \node at (-0.1, 0.4) {\large $b$};
        \node at (0.2, 0.12) {\large $c$};
        \node at (0.46, 0.4) {\large $d$};
        \node at (-0.25, 0.75) {\large $a$};
        \node at (0.65, 0.75) {\large $e$};
	\node at (0.2, -0.2) {$\boldsymbol{H}$};		
        \node at (2.3,0.5){	\renewcommand{\arraystretch}{2}
					\scriptsize \begin{tabular}{>{\columncolor[gray]{0.90}}c>{\columncolor[gray]{0.90}}c>{\columncolor[gray]{0.90}}c>{\columncolor[gray]{0.90}}c>{\columncolor[gray]{0.90}}c}
					 \rowcolor{amethyst!30!}   \textbf{$a$}& \textbf{$b$} &  \textbf{$c$} & \textbf{$d$} & \textbf{$e$}\\		
							0 & 0 & 0 & 0 & 0\\
							0 & 0 & 1 & 0 & 1\\
							0 & 1 & 0 & 0 & 1\\
							0 & 1 & 1 & 0 & 0\\
							1 & 0 & 0 & 1 & 1\\
							1 & 0 & 1 & 1 & 0\\
							1 & 1 & 0 & 1 & 0\\
							1 & 1 & 1 & 1 & 1\\
							
					\end{tabular}};
        \node at (2.3,-0.3) {$CA(8, H, 2)$};           
    \end{tikzpicture}	            
    \caption{A hypergraph $H$ and an optimal CA on $H$}
    \label{fig:eg}
\end{figure}
\end{center}
Raaphorst et al. \cite{raaphorst2018variable} showed that for a $t$-uniform hypergraph $H$ and positive integers $g$ and $n$, a covering array $CA(n, H, g)$ exists if and only if there is a hypergraph homomorphism from $H$ to $t\text{-}QI(n, g)$. This 
motivated us to do a detailed study on the structural properties of QI hypergraphs. \\

A \emph{subhypergraph} $H'=(V', E')$ of a hypergraph $H=(V,E)$ is a hypergraph such that $V'\subseteq V$, $E' \subseteq E \cap \mathcal{P}(V')$, where $\mathcal{P}(V')$ is the power set of $V'$. If $E' = E \cap \mathcal{P}(V')$, then $H'$ is said to be an \emph{induced} subhypergraph. A \emph{$t$-complete hypergraph of order $n$},  denoted $K_n^t$, is a $t$-uniform hypergraph $H=(V,E)$ such that $|V| = n$ and  $E$ consists of all the $t$-subsets of $V$. Given a hypergraph $H$, a \emph{$t$-clique} of size $m$ in $H$ is a subset $C \subseteq V$ of size $m$ such that the induced subhypergraph on $C$ is $K_m^t$. We write $\omega_t(H)$ to denote the size of the largest $t$-clique in $H$. For example, $\omega_3(H)=3$ for $H$ in Figure \ref{fig:eg}. For two hypergraphs $H=(V, E)$ and $H_1=(V_1,E_1)$, a \emph{hypergraph homomorphism} from $H$ to $H_1$, denoted $f : H \rightarrow H_1$, is a map $f : V \rightarrow V_1$ such that for every hyperedge $e = \{v_1,\ldots, v_m\} \in E$, there exists a hyperedge $e'\in E_1$ such that  the image $f(e) = \{f(v_1), \ldots,f(v_m)\} \subseteq e'$ and $|e|$ = $|f(e)|$. If $H=H_1$, then $f$ is known as an \emph{endomorphism} and additionally, if $f$ is bijective, it is known as an \emph{automorphism}.\\
  
A hypergraph $H$ is called a \emph{core} if every endomorphism of $H$ is an automorphism. For an arbitrary hypergraph $H$, a \emph{core of $H$}, denoted $H^\bullet$, is a subhypergraph of $H$ such that $H^\bullet$ is a core and there exists a homomorphism $H \rightarrow H^\bullet$. Every finite hypergraph has a unique core up to isomorphism \cite{micciancio2004using}. The core of a hypergraph is an interesting structural characteristic to study, as it is the minimal subhypergraph of a hypergraph that is homomorphically equivalent to itself. It preserves the size of the maximum $t$-cliques of a hypergraph as well as its strong chromatic number. The decision problem of CORE IDENTIFICATION, that is, ``given graphs $G_1$ and $G_2$, is $G_2$ the core of $G_1$?" is DP-complete \cite{fagin2003data} and the decision problem of CORE RECOGNITION, that is, ``given a graph $G$, is $G$ a core?" is co-NP-complete \cite{hell1992core}. Cores of graphs have been extensively studied from the 1960s under various names such as minimal graphs, reduced graphs, automorphic graphs, retract-rigid graphs, etc., in diverse contexts \cite{fellner1982minimal, hell1992core, bauslaugh1995core, nevsetvril1978classescore, royle2001algebraic}. Some well-known examples of cores include complete graphs $K_n$, odd cycles, odd wheels, Kneser graphs (including the Petersen graph), $k$-critical graphs, and the icosahedron, etc., whereas even cycles and complete graphs with one edge removed are not cores. Johnson graphs are either cores or have a complete core \cite{godsil2011cores}. For all positive integers $k$, $k$-uniform complete hypergraphs are cores. Figure \ref{corefig} shows $H^\bullet$, the core of $H$ in Figure \ref{fig:eg} such that there is a homomorphism $f:H \rightarrow H^\bullet$ defined by $f(a)=f(e)=e, f(b)=b, f(c)=c, f(d)=d$.\\

\begin{figure}
\centering
    \begin{tikzpicture}[scale=4]
         \draw[rotate around={150:(0.4, 0.4)}, black, fill=cyan!40, fill opacity=0.4] (0.4, 0.4) ellipse (0.15 cm and 0.55 cm);
         \draw [black, rounded corners=13mm, fill=green!40, fill opacity=0.4] (-0.3,0.5)--(0.7,0.5)--(0.2,-0.2)--cycle;
	\draw [black, rounded corners=13mm, fill=black!40, fill opacity=0.4] (-0.35,0.3)--(0.65,0.3)--(0.75,1)--cycle;
         \node[draw, black, circle,thick,minimum size=0.1cm,inner sep=0pt, fill=black] at (0.2, 0.05) {};
         \node[draw, black, circle,thick,minimum size=0.1cm,inner sep=0pt, fill=black] at (0.4, 0.4) {};
         \node[draw, black, circle,thick,minimum size=0.1cm,inner sep=0pt, fill=black] at (0.6, 0.7) {};
         \node[draw, black, circle,thick,minimum size=0.1cm,inner sep=0pt, fill=black] at (0, 0.4) {};
         \node at (-0.1, 0.4) {\large $b$};
         \node at (0.2, 0.12) {\large $c$};
         \node at (0.46, 0.4) {\large $d$};
         \node at (0.65, 0.75) {\large $e$};
         \node[draw, black, circle,thick,minimum size=0.1cm,inner sep=0pt, fill=black] at (2.2, 0.05) {};
         \node[draw, black, circle,thick,minimum size=0.1cm,inner sep=0pt, fill=black] at (2.55, 0.4) {};
         \node[draw, black, circle,thick,minimum size=0.1cm,inner sep=0pt, fill=black] at (2.6, 0.7) {};
         \node[draw, black, circle,thick,minimum size=0.1cm,inner sep=0pt, fill=black] at (2, 0.4) {};
         \node at (1.9, 0.4) {\large $b$};
         \node at (2.2, 0) {\large $c$};
         \node at (2.65, 0.4) {\large $d$};
         \node at (2.65, 0.75) {\large $e$};
         \draw (2.2,0.05)--(2.55,0.4)--(2.6,0.7)--(2,0.4)--(2.2,0.05);
         \draw (2.55,0.4)--(2,0.4);
         \draw (2.2,0.05)--(2.6,0.7);
         \node at (0.2, -0.25) {\boldsymbol{$H^\bullet$}};
         \node at (2.2, -0.25) {\boldsymbol{$[H^\bullet]_2 = K_4$}};
    \end{tikzpicture}
    \caption{The core $H^\bullet$ of the hypergraph $H$ from Figure \ref{fig:eg} and its $2$-section $[H^\bullet]_2$}
    \label{corefig}
\end{figure}
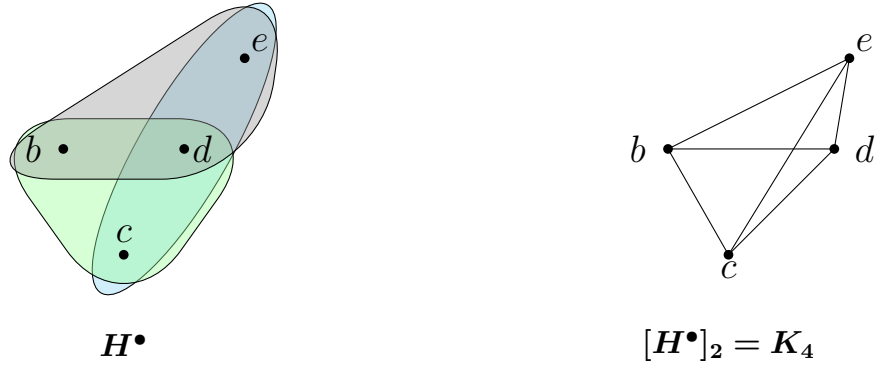

 Meagher et al. \cite{meagher2005covering} showed that the core of $2\text{-}QI(n, 2)$ is the complete graph on $\binom{n}{n/2}/2$ vertices, $K_{\binom{n}{n/2}/2}$ for even $n$, and it is the induced subgraph, where the vertices are those in which each symbol appears $\lfloor \frac{n}{2} \rfloor $ or $\lceil \frac{n}{2} \rceil $ many times for odd $n$. Godsil et al. \cite{godsilnewman33} showed that $2\text{-}QI(9,3)$ is a core. To the best of our knowledge, no study has been done on the cores of $t\text{-}QI(n, 2)$ for $t>2$ or of $2\text{-}QI(n, g)$ for $g\geq 3$ except when $n=9$. In this paper, we study the class of hypergraphs $3\text{-}QI(n,2)$ and trace their links to merged Johnson graphs, in order to study their cores. In particular, we show that $3\text{-}QI(8,2)$ is a core.\\
 
The rest of the paper is organised as follows. Section \ref{sec:2} sets the foundation for our results. In Section~\ref{sec:3}, we establish a structural connection between qualitative independence hypergraphs and merged Johnson graphs, with particular emphasis on their cores. Section~\ref{sec:4} focuses on the structural properties of $3\text{-}QI(8,2)$: we characterize its strongly independent sets and determine its strong independence number. In Section~\ref{sec:5}, we use these results and two invariants preserved under core reduction, namely its strong chromatic number and the size of its largest $3$-clique, to prove that $3\text{-}QI(8,2)$ is a core. Some concluding remarks and directions for future work are presented in Section~\ref{sec:6}.
\section{Preliminary Results}
\label{sec:2}

We first summarize the key properties of hypergraph cores, which form the basis for the results developed in the subsequent sections. 

\begin{lemma}
\label{induced}
If a hypergraph $H$ and its core $H^\bullet$ have the same order, then $H^\bullet = H$. More generally, the core of a hypergraph is an induced subhypergraph. 
\end{lemma}
\begin{proof} 
Since $H^\bullet$ is the core of $H$, there exists a homomorphism $f: H \rightarrow H^\bullet$. Suppose $H^\bullet=(V^\bullet, E^\bullet)$ is not induced and let $x_1, x_2, \ldots, x_n$ be vertices in $H^\bullet$ which do not form a hyperedge in $H^\bullet$ but form a hyperedge in $H$. Consider the induced subhypergraph $H^*=(V^*, E^*)$ of $H$ such that $V^*=V^\bullet$. Each hyperedge in $H^*$ is mapped to a hyperedge in $H^\bullet$ under $f$, and therefore, there are two distinct hyperedges of $H^*$ that map to the same hyperedge in $H^\bullet$. So, at least two distinct vertices are mapped to the same image vertex. This implies the existence of a subhypergraph of order less than that of $H^\bullet$, which is homomorphically equivalent to $H^*$, that is, there exists $g:H^*\rightarrow H^\bullet\backslash\{v\}$, for some $v \in H^\bullet$. This leads to the existence of homomorphisms $H \xrightarrow{f} H^\bullet \xrightarrow{i} H^* \xrightarrow{g}H^\bullet\backslash\{v\}$, where $i$ is the inclusion map. So we have a smaller subhypergraph $H^\bullet \backslash \{v\}$ of $H$ such that there is a homomorphism $g \circ i \circ f: H \rightarrow H^\bullet\backslash\{v\}$. This is a contradiction to $H^\bullet$ being the core of $H$. So, $H^\bullet$ has to be an induced subhypergraph. Now, if the number of vertices is the same in both $H$ and $H^\bullet$, this implies that $H=H^\bullet$.
\end{proof}

A hypergraph is \emph{vertex-transitive} if its automorphism group acts transitively on its vertices, that is, for any pair of vertices, there is an automorphism taking one vertex to the other. Godsil et al. \cite{royle2001algebraic} showed that the order of a vertex-transitive graph is divisible by the order of its core. We now establish the same property for any vertex-transitive hypergraph.

\begin{lemma}
\label{core}
If a hypergraph $H=(V, E)$ is vertex-transitive, then the order of its core $H^\bullet=(V^\bullet, E^\bullet)$ divides the order of $H$.
\end{lemma}
\begin{proof} 
If the fibres (pre-images) of vertices under a homomorphism $H \rightarrow H^\bullet$ have the same size, then the size of a fibre times $|V^\bullet|$ is $|V|$ and so, $|V^\bullet|$ divides $|V|$. The rest of the proof establishes that the fibres are of the same size.\\

Let $f:H \rightarrow H^\bullet$ be a homomorphism. Since $H^\bullet$ is the core of $H$, $f$ is surjective. Let $g \in Aut(H)$ and, therefore,  $|g(V^\bullet)|=|V^\bullet|$. Hence, $f(g(V^\bullet))=V^\bullet$ (if $X=f(g(V^\bullet))$ is a proper subhypergraph of $H^\bullet$, then it would be a smaller subhypergraph such that there exists the inclusion map $i:X \rightarrow H$ and a homomorphism $f\circ g \circ f: H \rightarrow X$, which contradicts the fact that $H^\bullet$ is the core). Hence, $g(V^\bullet)$ intersects each fibre of $f$ at exactly one vertex, that is, $f|_{g(V^\bullet)}$ is a bijection. \\

Suppose $v\in V$ and $F$ is the fibre of $f$ containing $v$. 
We claim that the number of automorphisms $g \in Aut(H)$ such that $g(V^\bullet)$ contains $v$, is independent of our choice of $v$. Let $g$ be an automorphism of $H$ such that $g(V^\bullet)$ contains $v$. So, there exists $v' \in H^\bullet$ such that $g(v')=v$. Let $w \neq v$ be a vertex in $H$. By vertex transitivity, there exists $g' \in Aut(H)$ such that $g'(v')=w$, that is, $g'$ is an automorphism of $H$ containing $w$ in its range. Thus, for any automorphism $g$ of $H$ such that the image of $H^\bullet$ under it contains vertex $v$, there exists an automorphism $g'$ of $H$ such that the image of $H^\bullet$ under it contains vertex $w$ and vice versa. That is, for all $v, w \in V$, there exists a one-to-one correspondence between the set of automorphisms of $H$ whose image of $H^\bullet$ contains $v$ and the set of automorphisms of $H$ whose image of $H^\bullet$ contains $w$. Let this number of automorphisms be $N$, that is, $N = |\{g:g \in Aut(H), v \in g(V^\bullet)\}|$, for any $v\in V$. Every image in $g(V^\bullet)$ meets $F$ at exactly one vertex since $f|_{g(V^\bullet)}$ is a bijection. \\

Now we compute the number of automorphisms in  $Aut(H)$. For a fixed vertex $v \in F$, $N$ is the number of automorphisms $g$ such that $v \in g(V^\bullet)$. Under such an automorphism, exactly one vertex from every other fibre also lies in $g(V^\bullet)$. The automorphisms not counted in $N$ are precisely those $g'$ for which some vertex of $F$ other than $v$ lies in $g'(V^\bullet)$. To account for these, we add $N$ for each vertex in $F$. Now, let $x$ be a vertex belonging to any fibre distinct from $F$, and let $h$ be an automorphism with $x \in h(V^\bullet)$. Then some vertex of $F$ must also lie in $h(V^\bullet)$, so this automorphism, $h$, has already been counted in the $|F|\cdot N$ automorphisms counted above. Therefore, $|Aut(H)| = |F| \cdot N$, and every fibre has the same size, namely $\frac{|Aut(H)|}{N}$.
\end{proof}

A \emph{uniform} $t$-qualitative independence hypergraph, $t\text{-}UQI(n, g)$, is an induced subhypergraph of $t\text{-}QI(n,g)$ whose vertices have the property that every symbol appears equally many times. Uniform partitions, corresponding to these vertices, reduce the search space required and are preferred for larger statistical power. Cliques in $t\text{-}UQI(n, g)$ correspond to \emph{balanced} CAs, that is, a CA with the property that each symbol occurs the same number of times in each column. Raaphorst et al. \cite{raaphorst2018variable} proved that a balanced CA on a hypergraph $H$ with size $n$ and $g$ symbols exists if and only if there exists a hypergraph homomorphism from $H$ to $t\text{-}UQI(n, g)$. An \emph{almost-uniform} $t$-qualitative independence hypergraph, $t\text{-}AUQI(n, g)$, is an induced subhypergraph of $t\text{-}QI(n,g)$ whose vertices have the property that if $n = gs+r$ for some integers $s$ and $0 < r < g$, then $(g-r)$ of the symbols appear $s$ times and the remaining $r$ appears $(s+1)$ times. Meagher et al. \cite{meagher2005covering} showed that the core of $2\text{-}QI(n, 2)$ is $2\text{-}UQI(n, 2)$, when $n$ is even, and $2\text{-}AUQI(n, 2)$, when $n$ is odd. In particular, $2\text{-}UQI(n, 2)$ and $2\text{-}AUQI(n, 2)$ are cores.
  
\begin{proposition} 
\label{vertextransitive}
The hypergraphs $3\text{-}UQI(n, 2)$ and $3\text{-}AUQI(n, 2)$ are vertex-transitive.
\end{proposition}
\begin{proof}
Let $H$ denote either $3\text{-}UQI(n,2)$ or $3\text{-}AUQI(n,2)$, depending on whether $n$ is even or odd. Since the vertices of $H$ are binary strings, we identify them with subsets of $\Omega=\{1,2,\ldots,n\}$ of size $k=\lfloor n/2\rfloor$. Each permutation $\sigma\in S_n$ acts on a vertex $A\subseteq \Omega$ by $A^\sigma := \{\, a^\sigma : a\in A \,\}$, 
which is again a $k$-subset of $\Omega$. Thus, every $\sigma\in S_n$ induces a permutation of the vertex set of $H$. If $A,B,C \subseteq \Omega$ form a hyperedge of $H$, then by definition, there exists at least one index corresponding to each possible ordered binary triple when $A, B, C$ are considered as strings. Therefore, all the eight triple intersections (with complements taken in $\Omega$) are nonempty. Since permutations preserve cardinalities and intersections, we have
\[
|A^\sigma \cap B^\sigma \cap C^\sigma| = |A \cap B \cap C|,
\]
and similarly for all the eight intersections. Hence $\{A^\sigma, B^\sigma, C^\sigma\}$ also forms a hyperedge of $H$. Thus, every $\sigma \in S_n$ induces an automorphism of $H$. To prove vertex-transitivity, let $A = \{i_1, \dots, i_k\}$ and $B = \{j_1, \dots, j_k\}$ be two distinct vertices. The bijection $i_m \mapsto j_m$ for $1 \le m \le k$ extends to a permutation $\pi \in S_n$. By the previous argument, $\pi$ induces an automorphism of $H$ and satisfies $A^\pi = B$. Thus, $H$ is vertex-transitive.
\end{proof} 
         
\noindent The following result is known about the size of the largest $t$-clique in the core of a hypergraph.

\begin{lemma}
\label{coreclique}\cite{raaphorst2013thesisvariable}
Let $H^\bullet$ be the core of a hypergraph $H$. Then, the size of the maximum $t$-clique, $\omega_t(H)=\omega_t(H^\bullet)$.
\end{lemma}
         
For a hypergraph $H=(V, E)$, a \emph{strong $k$-colouring} (of the vertices) is a $k$-partition $(S_1, S_2, \ldots S_k)$ of $V$, where vertices in $S_i$ for each $i$ are given the same colour such that no colour appears more than once within any hyperedge. Formally, for every hyperedge $e\in E$, $|e \cap S_i| \leq 1$ for all $i=1,2, \ldots,k$. The \emph{strong chromatic number} of a hypergraph $H$, denoted $\chi_S(H)$, is the smallest integer $k$ for which $H$ admits a strong $k$-colouring. For $H$ in Figure \ref{fig:eg}, a strong $4$-colouring is given by assigning colour 1 to $a$ and $e$, colour 2 to $b$, colour 3 to $c$, colour 4 to $d$, leading to $\chi_S(H)=4$.

\begin{lemma}
\cite{raina}
If $H$, $H'$ are hypergraphs such that there is a homomorphism $H \rightarrow H'$, then $\chi_S(H)\leq \chi_S(H')$.
\end{lemma}

\noindent The preceding lemma establishes that the core of a hypergraph preserves the strong chromatic number.
\begin{proposition}
\label{corechrom}
Let $H^\bullet$ be the core of a hypergraph $H$. Then, the strong chromatic number $\chi_S(H)=\chi_S(H^\bullet)$.
\end{proposition}


\section{Qualitative Independence Hypergraphs and Merged Johnson Graphs}
\label{sec:3}
Let $n, m, i$ be fixed positive integers with $n\geq m \geq i$ and denote $[n]=\{1,2, \ldots, n\}$. Define \emph{generalized Johnson graph} or \emph{uniform subset graph} or \emph{distance graph}, denoted $J(n, m, i)$, as follows: The vertices of $J(n, m, i)$ are the subsets of $[n]$ of size $m$, where two subsets are adjacent if their intersection has size $i$. Thus, $|V(J(n,m,i))|=\binom{n}{m}$ and $J(n,m,i)$ is $\binom{m}{i}\binom{n-m}{m-i}$-regular. If $I$ is any subset of $[m]$, the \emph{merged Johnson graph}, $J (n,m)_I$, is defined to be the edge-union of the graphs $J (n,m,i)$, where $i \in I$. Thus, $J (n,m)_I$ has as vertex set, all the $m$-subsets of $[n]$, with vertices $M$ and $M'$ adjacent if and only if $|M \cap M'| = i$ for some $i \in I$. Some well-known graphs in this class are Johnson graphs, Kneser graphs, including the Petersen graph, odd graphs, complete graphs, and line graphs of complete graphs.


\subsection{$2$-Section and Shadow Graphs}

We now establish a direct connection between qualitative independence hypergraphs and merged Johnson graphs. For this, we first define the graph derivatives: the shadow and the $k$-section. The \emph{shadow graph}, $D_2(G)$, of a connected graph $G$, is constructed by taking two copies $G_1$ and $G_2$ of $G$ and joining each vertex $u$ of $G_1$ to the neighbours of the corresponding vertex $u'$ in $G_2$. It can be observed that the vertices $u$ and $u'$ can be given the same colour in a proper vertex colouring of $D_2(G)$ and therefore $\chi(D_2(G))=\chi(G)$. Similarly, $\omega(D_2(G))=\omega(G)$. The shadow graph, $D_2(K_3)$ is illustrated in Figure \ref{fig:shadow}. The \emph{$k$-section} of a hypergraph $H=(V, E)$ is defined to be a hypergraph $[H]_k$ with the same vertex set and whose edges are the sets $F \subseteq V$ satisfying either $|F|=k$ and $F \subset e$ for some $e \in E$; or $|F|<k$ and $F=e$ for some $e \in E$. For instance, the $2$-section of $H$ in Figure \ref{fig:eg} is $K_5\backslash\{\{a,e\}, \{a,d\}\}$ and that of $H^\bullet$ in Figure \ref{corefig} is $K_4$, shown respectively in Figures \ref{fig:shadow} and \ref{corefig}.

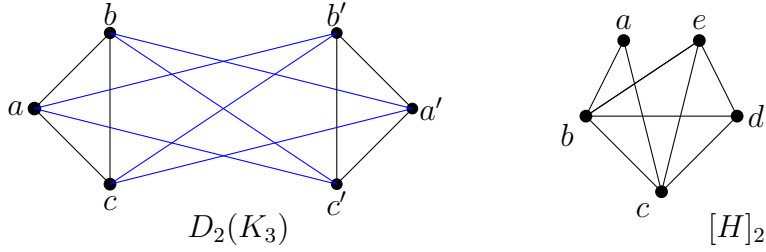
\begin{figure}
\centering
\begin{tikzpicture}
\node[draw, scale=0.6, black, circle,minimum size=0.1cm,inner sep=0pt, fill=black] at (-1,0) {a};
\node at (-1.25, 0) {$a$};
\node[draw, scale=0.4, black, circle,minimum size=0.1cm,inner sep=0pt, fill=black] at (0,1) {b};
\node at (0,1.25) {$b$};
\node[draw,scale=0.6,  black, circle,minimum size=0.1cm,inner sep=0pt, fill=black] at (0,-1) {c};
\node at (0, -1.25) {$c$};
\node[draw,  black, circle,minimum size=4pt,inner sep=0pt, fill=black] at (4,0) {};
\node at (4.25, 0) {$a'$};
\node[draw,  black, circle,minimum size=4pt,inner sep=0pt, fill=black] at (3,1) {};
\node at (3,1.25) {$b'$};
\node[draw,  black, circle,minimum size=4pt,inner sep=0pt, fill=black] at (3,-1) {};
\node at (3, -1.25) {$c'$};
\draw  (-1,0) -- (0,1);
\draw  (0,1) -- (0,-1);
\draw  (-1,0) -- (0,-1);
\draw (4,0) -- (3,1);
\draw (3,1) -- (3,-1);
\draw  (4,0) -- (3,-1);
\draw [ color=blue](-1,0) -- (3,1);
\draw [ color=blue] (-1,0) -- (3,-1);
\draw [ color=blue] (0,1) -- (4,0);
\draw [ color=blue] (0,1) -- (3,-1);
\draw [ color=blue] (0,-1) -- (4,0);
\draw [ color=blue] (0,-1) -- (3,1);
\node at (1.7,-1.6) {$D_2(K_3)$};
\end{tikzpicture}
\hspace{1cm}
\begin{tikzpicture}
\node[draw, black, circle,thick,minimum size=4pt,inner sep=0pt, fill=black] at (7, 0) {};
\node[draw, black, circle,thick,minimum size=4pt,inner sep=0pt, fill=black] at (8, 1) {};
\node[draw, black, circle,thick,minimum size=4pt,inner sep=0pt, fill=black] at (7.5, 2) {};
\node[draw, black, circle,thick,minimum size=4pt,inner sep=0pt, fill=black] at (6, 1) {};
\node[draw, black, circle,thick,minimum size=4pt,inner sep=0pt, fill=black] at (6.5, 2) {};
\node at (6.75, -0.25) { $c$};
\node at (8.25, 1) {$d$};
\node at (7.5, 2.25) {$e$};
\node at (5.75, 0.75) { $b$};
\node at (6.5, 2.25) { $a$};
\draw (7, 0)--(8, 1)--(7.5, 2)--(6, 1)--(6.5, 2);
\draw (7.5,2)--(7,0)--(6,1)--(7.5,2);
\draw (6.5,2)--(7,0);
\draw (6,1)--(8,1);
\node at (8, -0.5) {$[H]_2$};
\end{tikzpicture}
\caption{Shadow graph of $K_3$ and $2$-section of the hypergraph $H$ from  Figure \ref{fig:eg}}
\label{fig:shadow}
\end{figure}

\begin{proposition}
\label{endo}
Every endomorphism of a hypergraph $H$ is also an endomorphism of its $2$-section $[H]_2$.
\end{proposition}
\begin{proof}
Let $f$ be an endomorphism of $H$. If $\{a,b\}$ is an edge in $[H]_2$, then $a$ and $b$ lie together in some hyperedge of $H$. Since $f$ is an endomorphism of $H$, the vertices $f(a)$ and $f(b)$ also lie together in a hyperedge of $H$. Hence, $\{f(a),f(b)\}$ is an edge of $[H]_2$. Therefore, $f$ is an endomorphism of $[H]_2$.
\end{proof}  
             
\noindent The following propositions provide a tool for studying the core of $3$-QI hypergraphs by examining the core of their $2$-sections.

\begin{proposition}
\label{2section}
If $[H]_2$ is a core, then so is $H$.
\end{proposition}
\begin{proof}
Let $H^\bullet \subsetneq H$ be the core of $H$. Then, there exists a homomorphism $f : H \rightarrow H^\bullet$. By Proposition~\ref{endo}, $f$ is also a homomorphism $[H]_2 \rightarrow [H^\bullet]_2$. Since by Lemma~\ref{induced}, $H^\bullet$ is an induced subhypergraph of $H$,  the $2$-section $[H^\bullet]_2$ is an induced subgraph of $[H]_2$, and it is proper because $H^\bullet \subsetneq H$. Thus, we have a homomorphism from $[H]_2$ to a proper induced subgraph $[H^\bullet]_2$, contradicting the assumption that $[H]_2$ is a core. Hence, $H$ must also be a core.
\end{proof}

\begin{lemma}
\label{shadowcore}
Let $G$ be a graph and $D_2(G)$ be its shadow graph. Then, the core of $D_2(G)$ is isomorphic to a subgraph of $G$.
\end{lemma}
\begin{proof}
Suppose the core $D_2(G)^\bullet$ contains vertices from both copies of $G$ in $D_2(G)$. Let $f : D_2(G) \longrightarrow D_2(G)^\bullet$ 
be a homomorphism. For each vertex $u$ of $G$, let $u'$ denote its copy in the second copy of $G$ inside $D_2(G)$. Construct an induced subgraph $H$ of $D_2(G)$ as follows: for each vertex $x \in V(D_2(G)^\bullet)$, include $x$ if $x\in G$ otherwise include $u$ if $x=u'$ for some $u\in G$. Thus, $H$ is a subgraph of $G$. Define a map $f_1 : D_2(G) \rightarrow H$ by sending each vertex $v$ of $D_2(G)$ to the corresponding vertex in $H$ determined by $f(v)$. Now for any two vertices $u$ and $v$ adjacent in $D_2(G)$, $f(u)$ and $f(v)$ are adjacent in $D_2(G)^\bullet$ because $f$ is a homomorphism. If $f(u)$ and $f(v)$ lie in the same copy of $G$, then their corresponding vertices in $H$ are adjacent in $H$. If they lie in different copies, then by the definition of a shadow graph, the corresponding vertices in $H$ are also adjacent. Hence $f_1$ is a homomorphism. Since $H$ is a subgraph of $G$ and $|V(H)| \le |V(D_2(G)^\bullet)|$, and $D_2(G)^\bullet$ is minimal with respect to homomorphic equivalence, it follows that $H$ must be isomorphic to $D_2(G)^\bullet$. Therefore, the core of $D_2(G)$ is isomorphic to a subgraph of $G$.
\end{proof}

\begin{proposition}
\label{shadow}
 The core of a graph $G$ is isomorphic to the core of its shadow, $D_2(G)^\bullet$.
\end{proposition}
\begin{proof}
Let $G^\bullet$ be the core of $G$, and let $f : G \to G^\bullet$ be a homomorphism. We first extend $f$ to a map $f^* : D_2(G) \to G^\bullet$ by sending each vertex $u$ of $G$ and its copy $u'$ in $D_2(G)$ to the same image $f(u)$. 
Now, if $x$ and $y$ are adjacent in $D_2(G)$, then there are two possibilities:

\begin{enumerate}
\item[(i)] $x$ and $y$ lie in the same copy of $G$ (either the original or the duplicate).  
Then adjacency in $D_2(G)$ corresponds to adjacency in $G$, so $f(x)$ and $f(y)$ are adjacent in $G^\bullet$.   
\item[(ii)] $x=u$ lies in the original copy of $G$, and $y=v'$ is the copy of a vertex $v$ such that $u$ is adjacent to $v$ in $G$.  Since $f^*(v') = f(v)$, the vertices $f^*(u)$ and $f^*(v')$ are adjacent in $G^\bullet$.
\end{enumerate}

\noindent Therefore, $f^*$ is a homomorphism. Now suppose the core of $D_2(G)$, denoted $D_2(G)^\bullet$, is a proper subgraph of $G^\bullet$. By Lemma~\ref{shadowcore}, $D_2(G)^\bullet$ is isomorphic to a subgraph $H\subseteq G$ and hence there exists a homomorphism  $D_2(G) \to H$. Restricting this to the original copy of $G$ gives a homomorphism $G \to H$, contradicting the minimality of the core $G^\bullet$. Hence, $G^\bullet$ and $D_2(G)^\bullet$ must be isomorphic.
\end{proof}

\begin{corollary}
\label{nocoreshadow}
Let $G$ be a graph. Then, the shadow, $D_2(G)$, cannot be a core graph.
\end{corollary}
\begin{proof}
This follows from Lemma \ref{shadowcore}, since $G$ is always a proper subgraph of $D_2(G)$.
\end{proof}


\subsection{Cores of $3\text{-}UQI(n, 2)$ and $3\text{-}AUQI(n, 2)$}

The following is an important condition that characterizes the adjacency relation in $3\text{-}QI(n, 2)$.\\

\noindent\textbf{Condition 1}: A pair of distinct vertices $A$ and $B$ cannot belong to a hyperedge in $3\text{-}QI(n, 2)$ if and only if at least one of the four sets $A \cap B$, $A \cap \overline{B}$, $\overline{A} \cap B$, $\overline{A} \cap \overline{B}$ has cardinality at most one.
\begin{proof}
Let $X$ be one of the sets $A \cap B$, $A \cap \overline{B}$, $\overline{A} \cap B$, $\overline{A} \cap \overline{B}$ with cardinality less than two. Now, suppose that there is a hyperedge in $3\text{-}QI(n, 2)$ with vertices $A$, $B$, and $C$. Then, by definition of qualitative independence, each of $A\cap B\cap C$, $A\cap B \cap \overline{C}$, $A\cap \overline{B} \cap C$, $\overline{A}\cap \overline{B} \cap \overline{C}$, $\overline{A}\cap B \cap C$, $\overline{A}\cap \overline{B} \cap C$, $\overline{A}\cap B \cap \overline{C}$, $A\cap \overline{B} \cap \overline{C}$ must be non-empty. But here, either $X \cap C = \emptyset$ or $X \cap \overline{C}=\emptyset$, which contradicts the fact that vertices in a hyperedge in $3\text{-}QI(n, 2)$ are $3$-qualitatively independent. Conversely, if all these sets have cardinality at least two, then we can construct $C$ in the following manner. We take one element from each of $A \cap B$, $A \cap \overline{B}$, $\overline{A} \cap B$, $\overline{A} \cap \overline{B}$ to form $C$. The rest of the elements in $[n]$ form $\overline{C}$. This is certainly a vertex in $3\text{-}QI(n, 2)$ that forms a hyperedge with $A$ and $B$.
\end{proof}
              			 
\begin{lemma}
\label{uniform}
The shadow graph of the $2$-section of $3\text{-}UQI(n,2)$ is the merged Johnson graph, $J(n, \frac{n}{2})_{I}$, that is,\\$D_2([3\text{-}UQI(n,2)]_2)=J(n, \frac{n}{2})_{I}$, where $n$ is even and $I=\{2,3, \ldots, \frac{n}{2}-2\}$.
\end{lemma}	
\begin{proof}
The vertex set of $[3\text{-}UQI(n,2)]_2$ is the set of all $\frac{n}{2}$-subsets of $[n]$ containing $1$. In order to obtain the shadow graph, take two copies of these vertices and represent the second copy using their complement form. Together, they form $V(J(n,\frac{n}{2})_I)$. Within each copy that is in the family that contains $1$ and that which does not, if two subsets intersect at $i$ elements, $i \in I$, then by Condition 1, there will be an edge in $D([3\text{-}UQI(n,2)]_2)$. Now suppose $A$ is a vertex of $[3\text{-}UQI(n,2)]_2$, that is, it contains $1$, and let $B$ be a vertex in the family of subsets that do not have $1$. If $|A \cap B|=i$, where $i \in I$, then $|A \cap \overline{B}|=\frac{n}{2}-i$. So, $A$ and $\overline{B}$ are adjacent in $[3\text{-}UQI(n,2)]_2$ and hence, $A$ and $B$ are adjacent in $D([3\text{-}UQI(n,2)]_2)$. 
\end{proof}

\noindent This result leads to the following sufficient condition for $3\text{-}UQI(n, 2)$ to be a core.
\begin{proposition}\label{thm:evencore}
For even $n$ and $I=\{2, 3, \ldots, \frac{n}{2}-2\}$, the core of $[3\text{-}UQI(n, 2)]_2$ is isomorphic to the core of the merged Johnson graph $J(n, \frac{n}{2})_I$, which is a proper subgraph of $J(n, \frac{n}{2})_I$. Thus, if the core of $J(n, \frac{n}{2})_I$ has order $\frac{1}{2} \binom{n}{\frac {n}{2}}$, then $3\text{-}UQI(n, 2)$ is a core.
\end{proposition}
\begin{proof}
Lemma \ref{uniform} and Proposition \ref{shadow} imply that $[3\text{-}UQI(n, 2)]_2^\bullet$ is isomorphic to $J(n, \frac{n}{2})_I^\bullet$, where $I=\{2, 3, \ldots, \frac{n}{2}-2\}$. Since $J(n, \frac{n}{2})_I$ is the shadow graph of $[3\text{-}UQI(n, 2)]_2$, by Corollary \ref{nocoreshadow}, its core has to be a proper subgraph. If this proper subgraph has order $\frac{1}{2} \binom{n}{\frac {n}{2}} =|[V(3\text{-}UQI(n,2)]_2|$, then by Lemma~\ref{induced}, the core is $[3\text{-}UQI(n, 2)]_2$ itself. So, by Proposition \ref{2section}, $3\text{-}UQI(n, 2)$ is a core.
\end{proof}

Note that the condition in Proposition~\ref{thm:evencore} is not necessary for $3\text{-}UQI(n, 2)$ to be a core. In the following theorem, we show that $J(8, 4)_{\{2\}} = J(8, 4, 2)$, which has $\binom{8}{4} = 70$ vertices, has a core of order $7 (\neq \frac{1}{2} \binom{8}{4}=35)$. In contrast, in Section~\ref{sec:5}, we establish that $3\text{-}QI(8, 2) = 3\text{-}UQI(8, 2)$ is itself a core.

\begin{theorem}
\label{j842}
The core of the generalized Johnson graph $J(8, 4, 2)$ is the complete graph $K_7$.
\end{theorem}

\begin{proof}
We show that $J(8, 4, 2)$ contains a clique of size $7$ and that its chromatic number is $7$. This implies the existence of graph homomorphisms 
\[ K_7 \longrightarrow J(8, 4, 2) \longrightarrow K_7, \] 
thereby establishing that $K_7$ is the core of $J(8, 4, 2)$. The rest of the proof is devoted to proving that 
\[\chi(J(8, 4, 2)) = \omega(J(8, 4, 2)) = 7. \]
Since $J(8,4,2) = D_2([3\text{-}QI(8, 2)]_2)$, it follows that $\chi(J(8,4,2))=\chi([3\text{-}QI(8, 2)]_2)$ and $\omega(J(8,4,2)) = \omega([3\text{-}QI(8, 2)]_2)$. Therefore, we focus on $[3\text{-}QI(8, 2)]_2$, whose vertices are $2$-partitions $\{A, \overline{A}\}$ of $[8]$ with $|A| = |\overline{A}| = 4$.  Since $\overline{A}$ is uniquely determined by $A$, each vertex can be represented by a $4$-subset of $[8]$ containing the element $1$. There are total $\binom{7}{3}=35$ such vertices. Two vertices are adjacent if and only if their corresponding $4$-subsets intersect in exactly two elements. To generate a clique of size $7$, we consider the Fano plane, a finite projective plane of order $2$ on points $\mathcal{P}=\{2,3,4,5,6,7,8\}$ and lines $\mathcal{T}=\{l_1,\ldots,l_7\}$, where each line contains 3 points, each point lying on 3 lines, and each pair of points lies in exactly one line. Since any two distinct lines $l_i$ and $l_j$ in the Fano plane intersect in exactly one point, we have $|(\{1\}\cup l_i)\cap(\{1\}\cup l_j)|=2$. Hence, the vertices $\{1\}\cup l_i$ and $\{1\}\cup l_j$ are adjacent in $[3\text{-}QI(8, 2)]_2$, forming a clique of size $7$ and establishing $\omega([3\text{-}QI(8, 2)]_2)\ge7$.\\

To prove $\chi([3\text{-}QI(8, 2)]_2) = 7$, define
\[\mathcal{C}_i = \big\{\{1\}\cup l_i,\, \{1\}\cup f_{i1},\, \{1\}\cup f_{i2},\, \{1\}\cup f_{i3},\,\{1\}\cup f_{i4}\big\},\]
where each $f_{ij}$ ($j=1,2,3,4$) is a $3$-subset of the $4$-set $\mathcal{P} \setminus l_i$. Since $l_i \cap f_{ij} = \emptyset$, the sets $\{1\}\cup l_i$ and $\{1\}\cup f_{ij}$ intersect only in $\{1\}$, and hence are nonadjacent. Furthermore, for $j \ne k$, the subsets $f_{ij}$ and $f_{ik}$, being $3$-subsets of a $4$-set, intersect in two elements, implying $|(\{1\}\cup f_{ij}) \cap (\{1\}\cup f_{ik})| = 3$, so these vertices are also nonadjacent. Thus, each $\mathcal{C}_i$ ($i = 1, \ldots, 7$) forms an independent set. Because $|l_i \cap l_j| = 1$, it follows that $l_i \not\subseteq \mathcal{P} \setminus l_j$, and hence $\{1\}\cup l_i \notin \mathcal{C}_j$. Moreover, since $|l_i \cup l_j| = 5$, we have $|\mathcal{P} \setminus (l_i \cup l_j)| = 2$, implying that no $f_{ik}$ coincides with any $f_{jl}$. Therefore, the sets $\mathcal{C}_i$ and $\mathcal{C}_j$ are disjoint for $i \ne j$. As each $\mathcal{C}_i$ has $5$ vertices, their union consists of $35$ vertices, forming a partition of the vertex set of $[3\text{-}QI(8, 2)]_2$. Hence, each $\mathcal{C}_i$ serves as a colour class in a proper $7$-colouring of $[3\text{-}QI(8, 2)]_2$.  
Therefore,
\[\chi([3\text{-}QI(8, 2)]_2) = \omega([3\text{-}QI(8, 2)]_2) = 7.\]
\end{proof}

\begin{lemma}
\label{almostuniform}
The $2$-section of $3\text{-}AUQI(n,2)$ is the merged Johnson graph, $J(n, \frac{n-1}{2})_{I}$, that is, $[3\text{-}AUQI(n,2)]_2 = J(n, \frac{n-1}{2})_{I}$, where $n$ is odd and $I=\{2,3, \ldots, \frac{n-1}{2}-2\}$.
\end{lemma}
\begin{proof}
The vertex set here consists of all the $\frac{n-1}{2}$-subsets of $[n]$. By Condition 1, two vertices belong to a hyperedge if and only if they intersect at $i$ elements for some $i \in I$. 
\end{proof}
         
\noindent Thus, we also obtain a sufficient condition for $3\text{-}AUQI(n,2)$ to be a core.
\begin{proposition}\label{almostuniform2}
For odd $n$ and $I=\{2,3, \ldots, \frac{n-1}{2}-2\}$, if the merged Johnson graph, $J(n, \frac{n-1}{2})_{I}$, is a core, then so is $3\text{-}AUQI(n, 2)$.
\end{proposition}
\begin{proof}
Let $J(n, \frac{n-1}{2})_{I}$, where $I=\{2,3, \ldots, \frac{n-1}{2}-2\}$, be a core. This implies, by Lemma \ref{almostuniform}, that $[3\text{-}AUQI(n,2)]_2$ is a core and hence by Proposition \ref{2section}, $3\text{-}AUQI(n,2)$ is a core.
\end{proof}
We now integrate existing results on the automorphism groups of merged Johnson graphs with our observation to obtain a result on the core of $3\text{-}AUQI(n, 2)$. To proceed, we first recall some standard notation from group theory.\\

A group $(G, *)$ is said to \emph{act on a set $X$} when there is a map $\phi:G \times X \rightarrow X$ such that the following conditions hold for all elements $x \in X$.
\begin{enumerate}
    \item $\phi(e,x)=x$ where $e$ is the identity element of $G$.
    \item $\phi(g,\phi(h,x))=\phi(g*h,x)$ for all $g,h \in G$.
\end{enumerate}
 Here, $\phi$ is known as the \emph{group action}. The \emph{orbit} of an element $x \in X$ is the set $\{\phi(g, x) : g \in G\}$. Orbits of a group $G$ on $X \times X$ are called \emph{orbitals}. The automorphism group of merged Johnson graphs is discussed in detail by G A Jones in \cite{jones2005automorphisms} proving that the automorphism group of $J(n, \frac{n-1}{2})_I$, denoted $Aut(J(n, \frac{n-1}{2})_I)$, where $I=\{2,3, \ldots \frac{n-1}{2} -2 \}$, is $S_n$ with orbitals $\Gamma_0, \ldots, \Gamma_{\frac{n-1}{2}}$, where \[\Gamma_i = \{(M, M') \in \Omega^2 : |M \cap M'| = i \},\] $S_n$ is the permutation group on $n$ symbols, and $\Omega$ is the set of all $\frac{n-1}{2}$-subsets of $[n]$. These insights allow us to formulate the following sufficient condition for determining when $3\text{-}AUQI(n,2)$ is a core. 

\begin{corollary}\label{final2}
For odd $n$ and $I=\{2,3, \ldots \frac{n-1}{2} -2 \}$, the hypergraph $3\text{-}AUQI(n, 2)$ is a core, if the endomorphism monoid of $J(n, \frac{n-1}{2})_I$, denoted $End(J(n, \frac{n-1}{2})_I)$, is $S_n$ with orbitals $\Gamma_0, \ldots, \Gamma_{\frac{n-1}{2}}$.
\end{corollary}
\begin{proof}
    If $End(J(n, \frac{n-1}{2})_I)=Aut(J(n, \frac{n-1}{2})_I)$, then $J(n, \frac{n-1}{2})_I$ is a core. This implies, by Proposition \ref{almostuniform2} that $3\text{-}AUQI(n, 2)$ is a core.
\end{proof}

\section{Structural Properties of $3\text{-}QI(8,2)$}
\label{sec:4}

In this section, we list the strong chromatic number and the size of the largest $3$-cliques of of $3$-$QI(8,2)$. In addition, we determine the structure of the maximum strongly independent sets of $3$-$QI(8,2)$. These results will be used in Section~\ref{sec:5} to establish that the hypergraph $3$-$QI(8,2)$ is a core.

\begin{proposition}\label{chrom} 
The strong chromatic number of $3$-$QI(8,2)$ is $\chi_S(3$-$QI(8,2))=7$.
\end{proposition}
\begin{proof}
From Theorem~\ref{j842}, we know that $\chi([3$-$QI(8,2)]_2)=7$. Since for any hypergraph $H$, $\chi_S(H)=\chi([H]_2)$~\cite{berge1984hypergraphs}, the result follows directly.
\end{proof}

Next, we examine the size of the largest $3$-cliques in $3$-$QI(8,2)$. Since $\omega([3$-$QI(8,2)]_2)=7$, we have $\omega_3(3$-$QI(8,2))\leq 7$. To determine the exact value of $\omega_3(3$-$QI(8,2))$, we use the connection between $t$-cliques in $t$-$QI(n,g)$ and the columns of a covering array of strength $t$ with $g$ symbols. 

\begin{proposition}\label{clique}\cite{raina}
The size of the largest $3$-clique in $3$-$QI(8,2)$ is $\omega_3(3$-$QI(8,2))=4$.
\end{proposition}

We now study the strongly independent sets in $3$-$QI(8,2)$. A set of vertices in a hypergraph $H$ is said to be \emph{strongly independent} if no two belong to a common hyperedge in $H$. The maximum cardinality of such a set is the \emph{strong independence number}, denoted $\alpha_S(H)$. For example, in the hypergraph $H$ from Figure~\ref{fig:eg}, $\{a,e\}$ is a strongly independent set and $\alpha_S(H)=2$.  
As shown in Theorem~\ref{j842}, the vertices of $3$-$QI(8,2)$ are $4$-subsets of $[8]$, each containing the element $1$. For brevity, we represent each vertex by a string of length 4.

\begin{theorem}\label{types}
Let $a,b,c,d,e,f,g,h$ be distinct elements of $[8]$.  
The only possible maximum strongly independent subsets of $3$-$QI(8,2)$ containing $A=abcd$ are of the following three types:
\begin{itemize}
    \item \textbf{Type I:} $A,A_1,A_2,A_3,A_4$
    \item \textbf{Type II:} $A, A_i,B_{1i},B_{2i},B_{3i}$, for $i=1,2,3,4$
    \item \textbf{Type III:} $A,B_{j1},B_{j2},B_{j3},B_{j4}$, for $j=1,2,3$
\end{itemize}
where the vertices $A_i$ and $B_{ij}$ are listed in Table~\ref{tab:ind}.
\begin{table}[htbp]
\centering
\caption{Vertices that do not belong to any hyperedge containing $A=abcd$ in $3$-$QI(8,2)$}
\begin{tabular}{|c|c|c|c|}
\hline
$A_1=afgh$ & $B_{11}=acde$ & $B_{21}=abde$ & $B_{31}=abce$ \\ \hline
$A_2=aegh$ & $B_{12}=acdf$ & $B_{22}=abdf$ & $B_{32}=abcf$ \\ \hline
$A_3=aefh$ & $B_{13}=acdg$ & $B_{23}=abdg$ & $B_{33}=abcg$ \\ \hline
$A_4=aefg$ & $B_{14}=acdh$ & $B_{24}=abdh$ & $B_{34}=abch$ \\ \hline
\end{tabular}
\label{tab:ind}
\end{table}
\end{theorem}
\begin{proof}
Without loss of generality, let $a=1$.  
Since $n=8$, for any vertex $X$, if $|A \cap X|=3$, then $|\overline{A} \cap X|=1$.  
Every vertex of $3$-$QI(8,2)$ contains $a=1$, so a vertex not belonging to any hyperedge containing $A$ must intersect $A=abcd$ in either one or three elements.  
Hence, there are $\binom{4}{3}+4\binom{3}{2}=16$ such vertices, listed in Table~\ref{tab:ind}.  

The first column contains vertices intersecting $A$ only in $\{a\}$, while the remaining columns list those intersecting $A$ in exactly three elements.  
Each column forms a strongly independent set since any two vertices within a column intersect in three elements.  
Together with $A$, these yield the Type~I and Type~III sets.  
Similarly, each row forms a strongly independent set since the vertices in a row intersect pairwise in one or three elements, giving the Type~II sets.  
Finally, no two vertices from different rows and columns can appear together in a strongly independent set since such pairs intersect in exactly two elements.  
Thus, the only possible maximum strongly independent subsets of $3$-$QI(8,2)$ containing $A$ are precisely those of Types~I, II, and III.
\end{proof}

\begin{corollary}\label{ind}
The strong independence number of $3$-$QI(8,2)$, $\alpha_S(3$-$QI(8,2))=5$.
\end{corollary}

\section{The Core of $3$-$QI(8,2)$}
\label{sec:5}

In this section, we examine the structure of strongly independent sets arising from strong colourings and endomorphisms of $3$-$QI(8,2)$, which plays a crucial role in establishing that the hypergraph is a core. In particular, we show that in any $7$-strong colouring of $3$-$QI(8,2)$, for every choice of three colour classes, there exist three vertices that form a hyperedge. These colour classes may correspond to any of the Types~I, II, or III defined in Theorem~\ref{types}, with all of the distinct types or at least two of the same type. To facilitate this analysis, we first state a condition that characterizes adjacency in $3$-$QI(8,2)$. This condition serves as a key criterion for determining whether a given triple of vertices forms a hyperedge in the lemmas that follow.\\

\noindent\textbf{Condition 2}:  Let $A, B, C$ be vertices represented as $4$-subsets in $3\text{-}QI(8,2)$. Then they form a hyperedge if and only if $|A\cap B|=|A\cap C|=|B \cap C|=2$, and $|A\cap B\cap C|=1$. 

\begin{lemma}
\label{lemma1}
Consider a partition of the vertices of $3$-$QI(8,2)$ into classes of size five, each being a maximum strongly independent set. Among any three such classes, if two are of Type~I, then there exist three vertices, one from each class, that together form a hyperedge.
\end{lemma}

\begin{proof}
Let $\mathcal{C}_1, \mathcal{C}_2,$ and $\mathcal{C}_3$ be three classes, with $\mathcal{C}_2$ and $\mathcal{C}_3$ of Type~I. Since $\mathcal{C}_i$'s are maximum strongly independent sets, for any vertex $A\in\mathcal{C}_1$, there are vertices in $\mathcal{C}_2$ and $\mathcal{C}_3$ that intersect $A$ in exactly two elements. Without loss of generality, let $A=abcd \in \mathcal{C}_1$ and $abef \in \mathcal{C}_2$ intersecting $A$ in exactly two elements.  \\

The first row of Table~\ref{tab:TypeI} represents a Type~I class $\mathcal{C}_2$ containing the vertex $abef$.  
Each of the remaining rows corresponds to a possible form of Type~I class $\mathcal{C}_3$, depending on whether its first vertex intersects $abef$ in one, two, or three elements.  
The first entries in rows 2–6 intersect $A$ in exactly two elements.  
For reference:  
$acgh$ intersects $abef$ in one element ($a$);  
$aceg$ in two elements ($ae$ that is distinct from the intersection $ac$ with $A$);  
$abgh$ in the same two elements as with $A$ ($ab$);  
$abeg$ in three (two elements $ab$ shared with $A$);  
and $acef$ in three (with one shared element $a$ with $A$).

\begin{table}[htbp]
\centering
\caption{Type I classes $\mathcal{C}_2$ and $\mathcal{C}_3$}
\renewcommand{\arraystretch}{1}
\begin{tabular}{|c|c|c|c|c|c|}
\hline
$\mathcal{C}_2$&$abef$ & $acdg$ & $acdh$ & $adgh$ & $acgh$ \\\hline
\multirow{5}{*}{$\mathcal{C}_3$}
&$acgh$ & $abde$ & $abdf$ & $abef$ & $adef$ \\\hhline{~-----}
&$aceg$ & $abdf$ & $abdh$ & $abfh$ & $adfh$ \\\hhline{~-----}
&$abgh$ & $acde$ & $acdf$ & $adef$ & $acef$ \\\hhline{~-----}
&$abeg$ & $acdf$ & $acdh$ & $adfh$ & $acfh$ \\\hhline{~-----}
&$acef$ & $abdg$ & $abdh$ & $adgh$ & $abgh$ \\\hline
\end{tabular}
\label{tab:TypeI}
\end{table}

None of the rows~2, 5, or 6 of Table~\ref{tab:TypeI} can serve as $\mathcal{C}_3$, since each shares a vertex with the class $\mathcal{C}_2$ shown in row~1 (namely $abef$, $acdh$, and $adgh$, respectively). Hence, they cannot occur in a partition of the vertex set containing the class $\mathcal{C}_2$. If $\mathcal{C}_3$ is as in row~3, vertices $A=abcd$, $abef$, and $aceg$ form a hyperedge.  
If $\mathcal{C}_3$ corresponds to row~4, the structure of $\mathcal{C}_1$ determines the resulting hyperedge as follows:

\begin{enumerate}
    \item \textbf{$\mathcal{C}_1$ of Type~I:} By Theorem~\ref{types}, $\mathcal{C}_1$ must be the following maximum strongly independent set.
    \begin{center}
    \begin{tabular}{|c|c|c|c|c|}
    \hline
    $abcd$ & $aefg$ & $aefh$ & $afgh$ & $aegh$ \\
    \hline
    \end{tabular}
    \end{center}
    In this case, the vertices $aefg$, $acgh$, and $acde$ form a hyperedge.

    \item \textbf{$\mathcal{C}_1$ of Type~II:} By Theorem~\ref{types}, $\mathcal{C}_1$ has the structure: 
    \begin{center}
    \begin{tabular}{|c|c|c|c|c|}
    \hline
    $abcd$ & $axyz$ & $abcw$ & $abdw$ & $acdw$ \\
    & where $x,y,z\notin\{b,c,d\}$ and $w\notin\{b,c,d,x,y,z\}$ & & & \\
    \hline
    \end{tabular}
    \end{center}
    If $xyz = efg$ or $efh$, then $aefg$ or $aefh$, respectively, forms a hyperedge with $acgh$ and $acde$.  
    If $xyz = egh$ or $fgh$, then $aegh$ or $afgh$ forms a hyperedge with $acdg$ and $acef$.

    \item \textbf{$\mathcal{C}_1$ of Type~III:} By Theorem~\ref{types}, $\mathcal{C}_1$ has the structure: 
    \begin{center}
    \begin{tabular}{|c|c|c|c|c|}
    \hline
    $abcd$ & $axye$ & $axyf$ & $axyg$ & $axyh$ \\
    & where $x,y \in \{b,c,d\}$ and $x \neq y$ & & & \\
    \hline
    \end{tabular}
    \end{center}
    If $xy = bc$ or $bd$, then $axye$, $acdg$, and $abgh$ form a hyperedge.  
    If $xy = cd$, then $acde$ repeats, contradicting that $\mathcal{C}_1$ and $\mathcal{C}_3$ are disjoint.
\end{enumerate}
\noindent Hence, in every valid configuration where two classes are of Type~I, a hyperedge exists among the three selected classes.
\end{proof}

\begin{lemma}
\label{lem2}
Consider a partition of the vertices of $3$-$QI(8,2)$ into classes of size five, each being a maximum strongly independent set. Among any three such classes, if two are of Type~II, then there exist three vertices, one from each class, that together form a hyperedge.
\end{lemma}
\begin{proof}
Let $A=abcd$ be a vertex in a maximum strongly independent set $\mathcal{C}_1$, and let $\mathcal{C}_2$ and $\mathcal{C}_3$ be maximum strongly independent sets of Type~II. The maximality of $\mathcal{C}_2$ and $\mathcal{C}_3$ implies that they contain at least one vertex that intersects $A$ in exactly two elements. Without loss of generality, let $abef \in \mathcal{C}_2$ be one such vertex. The first row of Table~\ref{tab:TypeII} represents a Type~II class $\mathcal{C}_2$ containing the vertex $abef$, where $w,x,y,$ and $z$ are distinct elements from the set $\{c,d,g, h\}$. Each of the remaining rows corresponds to a possible form of Type~II class $\mathcal{C}_3$, depending on whether its first vertex intersects $abef$ in one, two, or three elements. Similar to Lemma~\ref{lemma1}, the first entries in rows 2–6 intersect $A$ in exactly two elements. \\
\begin{table}[h!]
\centering
\caption{Type II classes $\mathcal{C}_2$ and $\mathcal{C}_3$}
    \begin{tabular}{|c|c|c|c|c|c|}	
    \hline
    $\mathcal{C}_2$&$abef$&$axyz$&$abew$&$abfw$&$aefw$\\\hline
     \multirow{5}{*}{$\mathcal{C}_3$}&$acgh$&$ax_1y_1z_1$&$acgw_1$&$achw_1$&$aghw_1$\\\hhline{~-----}
     &$aceg$&$ax_2y_2z_2$&$acew_2$&$acgw_2$&$aegw_2$\\\hhline{~-----}
     &$abgh$&$ax_3y_3z_3$&$abgw_3$&$abhw_3$&$aghw_3$\\\hhline{~-----}
     &$abeg$&$ax_4y_4z_4$&$abew_4$&$abgw_4$&$aegw_4$\\\hhline{~-----}
     &$acef$&$ax_5y_5z_5$&$acew_5$&$acfw_5$&$aefw_5$\\\hline
     \end{tabular}
     \label{tab:TypeII}
     \end{table}
        
Suppose $\mathcal{C}_3$ corresponds to row~2 of Table~\ref{tab:TypeII}, where $w_1,x_1,y_1,z_1$ are distinct elements of $\{b,d,e,f\}$.  
The array below shows that, for any choice of $w,x,y,z$ and $w_1,x_1,y_1,z_1$, either two of the classes fail to be disjoint or there exists a hyperedge containing one vertex from each of $\mathcal{C}_1,\mathcal{C}_2,$ and $\mathcal{C}_3$.  
\begin{center}
\begin{tabular}{|c|c|c|c|c|c|}
\hline
    $xyz$ & $w$ & $x_1y_1z_1$& $w_1$& Hyperedge with $A=abcd$& $\mathcal{C}_2 \cap \mathcal{C}_3$\\
    \hline
    &&$bef$&&&$abef$\\\hline
    $cdg$ &$h$& $bde$&$f$&$abef, acfg$&\\\hline
    $cdg$ &$h$& $bdf$&$e$& $abef, aceg$&\\\hline
    $cdg$ &$h$& $def$&$b$& $abeh, acgh$&\\\hline
    $cdh$ & $g$ &$bde$&$f$& $abef, acfg$&\\\hline
    $cdh$ & $g$ &$bdf$&$e$& $abef, aceg$&\\\hline
    $cdh$ & $g$ &$def$&$b$& $abeg, acgh$&\\\hline
    $cgh$& $d$& &&&$acgh$\\\hline
    $dgh$& $c$& $bde$&$f$&$abef, acfg$&\\\hline
    $dgh$& $c$& $bdf$&$e$&$abef, aceg$&\\\hline
    $dgh$& $c$& $def$&$b$&(separately discussed in Case 1)&\\\hline
\end{tabular}
\end{center}
\begin{enumerate}
\item[Case 1.]                
If $xyz=dgh$ and $x_1y_1z_1=def$,  the structure of $\mathcal{C}_1$ determines the resulting hyperedge as follows:
    \begin{enumerate}
    \item \textbf{$\mathcal{C}_1$ of Type~I:} By Theorem~\ref{types}, $\mathcal{C}_1$ has the following structure. \\
        \begin{center}
        \begin{tabular}{|c|c|c|c|c|}
        \hline
	$abcd$& $aefg$& $aefh$& $afgh$& $aegh$\\
        \hline
        \end{tabular}
        \end{center}
        
    In this case, $aefg\in \mathcal{C}_1$, $abce \in \mathcal{C}_2$, and $acgh\in \mathcal{C}_3$ form a hyperedge.
    \item \textbf{$\mathcal{C}_1$ of Type~II:} By Theorem~\ref{types}, $\mathcal{C}_1$ has the structure: \\
	\begin{center}                             
        \begin{tabular}{|c|c|c|c|c|}
        \hline
	$abcd$& $axyz$,& $abcw$,& $abdw$& $acdw$\\
              &where $x,y,z\notin\{b,c,d\}$&where $w\notin \{b,c,d,x,y,z\}$&&\\
        \hline
        \end{tabular}
	\end{center}  
    \smallskip
    Then, the array below shows that, for any choices of $w,x,y,z$ either two of the classes fail to be disjoint or there exists a hyperedge containing one vertex from each of the three classes.\\
        \begin{center}
        \begin{tabular}{|c|c|c|}
        \hline
        $xyz$  & Hyperedge & $\mathcal{C}_1 \cap \mathcal{C}_2$ or $\mathcal{C}_1 \cap \mathcal{C}_3$ \\
        \hline
        $efg$  & $aefg, abce, acgh$&\\\hline
        $efh$&$aefh, abce, acgh$&\\\hline
        $egh$&&$abcf$\\\hline
        $fgh$&&$abce$\\\hline
        \end{tabular}
        \end{center}
        \smallskip
    \item \textbf{$\mathcal{C}_1$ of Type III:}  By Theorem~\ref{types}, $\mathcal{C}_1$ has the following structure. \\
	\begin{center}
        \begin{tabular}{|c|c|c|c|c|}
        \hline
	$abcd$& $axye$,& $axyf$& $axyg$& $axyh$\\
              &where $x,y \in \{b, c, d\}=b$ and $x \neq y$&&&\\
        \hline
        \end{tabular}
        \end{center}
        \smallskip
    If $xy=bc$, then $abce \in \mathcal{C}_1\cap \mathcal{C}_2$. If $xy=bd$ or $cd$, then $axye \in \mathcal{C}_1$, $adgh \in \mathcal{C}_2$ and $abcg \in \mathcal{C}_3$ form a hyperedge. 
    \end{enumerate}
\end{enumerate}
Thus, in every admissible choice, the conclusion of the lemma holds when $\mathcal{C}_3$ is given by row~2 of Table~\ref{tab:TypeII}. Similarly, if $\mathcal{C}_3$ corresponds to row~3 of Table~\ref{tab:TypeII}, then $abef$ and $aceg$ form a hyperedge with $A$. So, next, suppose $\mathcal{C}_3$ corresponds to row~4 of Table~\ref{tab:TypeII}, where $w_3,x_3,y_3,$ and $z_3$ are distinct elements from the set $\{c,d,e, f\}$. The array below shows that, for any choice of $w,x,y,z$ and $w_3,x_3,y_3,z_3$, either two of the classes share a common vertex or there exists a hyperedge containing one vertex from each of $\mathcal{C}_1,\mathcal{C}_2,$ and $\mathcal{C}_3$. 
\begin{center}
\begin{tabular}{|c|c|c|c|c|c|}
    \hline
    $xyz$ & $w$ & $x_3y_3z_3$& $w_3$&Hyperedge with $A=abcd$& $\mathcal{C}_2 \cap \mathcal{C}_3$\\\hline
    $cdg$ &$h$& $cde$&$f$& &$abfh$\\\hline
    $cdg$ &$h$& $cdf$&$e$& &$abeh$\\\hline
    $cdg$ &$h$& $cef$&$d$& $abeh, acef$&\\\hline
    $cdg$ &$h$& $def$&$c$& $abeh, adef$&\\\hline
    $cdh$ & $g$ &$cde$&$f$& &$abfg$\\\hline
    $cdh$ & $g$ &$cdf$&$e$& &$abeg$\\\hline
    $cdh$ & $g$ &$cef$&$d$& $abeg, acef$&\\\hline
    $cdh$ & $g$ &$def$&$c$& $abeg, adef$&\\\hline
    $cgh$& $d$& $cde$&$f$&$acgh, abfh$&\\\hline
    $cgh$& $d$& $cdf$&$e$&$acgh, abeh$&\\\hline
    $cgh$& $d$& $cef$&$d$&(separately discussed in Case 2)&\\\hline
    $cgh$& $d$& $def$&$c$&&$acgh$\\\hline
    $dgh$& $c$& $cde$&$f$&$adgh, abfh$&\\\hline
    $dgh$& $c$& $cdf$&$e$&$adgh, abeh$&\\\hline
    $dgh$& $c$& $cef$&$d$&&$adgh$\\\hline
    $dgh$& $c$& $def$&$c$&(separately discussed in Case 3)&\\\hline
\end{tabular}
\end{center}
                            
\begin{enumerate}
\item[Case 2.] If $xyz=cgh$ and $x_3y_3z_3=cef$, then the structure of $\mathcal{C}_1$ given by Theorem~\ref{types} leads to the following conclusions.
    \begin{enumerate}
    \item \textbf{$\mathcal{C}_1$ of Type I:} In this case, $\mathcal{C}_1$ is				\begin{center} 
        \begin{tabular}{|c|c|c|c|c|}
        \hline
	$abcd$& $aefg$&$aefh$&$afgh$&$aegh$\\\hline
        \end{tabular}
        \end{center}
        \smallskip
    which results in a hyperedge formed by $aefg$, $abde$, and $adgh$.
    \item \textbf{$\mathcal{C}_1$ of Type II:}
    In this case, $\mathcal{C}_1$ has the structure:
	\begin{center}
        \begin{tabular}{|c|c|c|c|c|}
        \hline
	$abcd$& $axyz$,& $abcw$,& $abdw$& $acdw$\\
        &where  $x,y,z\notin \{b,c,d\}$&where $w\notin \{b,c,d,x,y,z\}$&&\\\hline
        \end{tabular}
        \end{center}
        \smallskip
    Then, the array below shows that, for any choices of $w, x,y,z$, either two of the classes share a common vertex or there exists a hyperedge containing one vertex from each of the classes.\smallskip
    
        \begin{center}
        \begin{tabular}{|c|c|c|}
        \hline
        $xyz$ & Hyperedge & $\mathcal{C}_1 \cap \mathcal{C}_2$ or $\mathcal{C}_1 \cap \mathcal{C}_3$ \\\hline
        $efg$& $aefg, abde, adgh$&\\\hline
        $efh$& $aefh, abde, adgh $&\\\hline
        $egh$&&$abdf$\\\hline
        $fgh$&& $abde$\\\hline
        \end{tabular}
        \end{center}
        \smallskip
    \item \textbf{$\mathcal{C}_1$ of Type III:} Here $\mathcal{C}_1$ has the following structure. 
	\begin{center}
        \begin{tabular}{|c|c|c|c|c|}
        \hline
	$abcd$& $axye$, &$axyf$&$axyg$&$axyh$\\
        &where $x,y\in \{b, c, d\}$ and $x \neq y$&&&\\
        \hline
        \end{tabular}
	\end{center}
    \smallskip
    If $xy=bd$, then $abde \in \mathcal{C}_1\cap \mathcal{C}_2$. If $xy=bc$ or $cd$, then $axye$, $acgh$, and $abdg$ form a hyperedge. \\
    \end{enumerate}
\item[Case 3.] 
If $xyz=dgh$ and $x_3y_3z_3=def$, then again the structure of $\mathcal{C}_1$ given by Theorem~\ref{types} leads to the following conclusions.
    \begin{enumerate}
    \item \textbf{$\mathcal{C}_1$ of Type I:} Here $\mathcal{C}_1$ is
        \begin{center} 
        \begin{tabular}{|c|c|c|c|c|}
        \hline
	$abcd$& $aefg$& $aefh$& $afgh$& $aegh$\\\hline
        \end{tabular}
        \end{center}
        \smallskip
    which results in a hyperedge formed by $aefg$, $abce$, and $abgh$.
    \item \textbf{$\mathcal{C}_1$ of Type II:} In this case, $\mathcal{C}_1$  has the structure:
	\begin{center}
        \begin{tabular}{|c|c|c|c|c|}
        \hline
	$abcd$& $axyz$,& $abcw$,& $abdw$& $acdw$\\
            &where  $x,y,z\notin\{b,c,d\}$&where $w\notin \{b,c,d,x,y,z\}$&&\\\hline
        \end{tabular}
	\end{center}
    \smallskip
    Then, the array below shows that, for any choices of $w, x,y,z$, either two of the classes are not disjoint or there exists a hyperedge containing one vertex from each of the classes.
        \begin{center}
        \begin{tabular}{|c|c|c|}
        \hline
        $xyz$ & Hyperedge & $\mathcal{C}_1 \cap \mathcal{C}_2$ or $\mathcal{C}_1 \cap \mathcal{C}_3$ \\\hline
        $efg$& $aefg, abce, acgh$&\\\hline
        $efh$& $aefh, abce, acgh $&\\\hline
        $egh$&&$abcf$\\\hline
        $fgh$&& $abce$\\\hline
        \end{tabular}
        \end{center}
        
        \smallskip
    \item \textbf{$\mathcal{C}_1$ of Type III:} In this case, $\mathcal{C}_1$  has the structure:
	\begin{center}
        \begin{tabular}{|c|c|c|c|c|}
        \hline
	$abcd$& $axye$,&$axyf$&$axyg$&$axyh$\\
            &where $x, y\in\{b, c, d\}$ and $x \neq y$&&&\\\hline
        \end{tabular}
	\end{center}
    \smallskip
    If $xy=bc$, then $abce\in \mathcal{C}_1\cap \mathcal{C}_2$. If $xy=bd$ or $cd$, then $axye$, $adgh$ and $abcg$ form a hyperedge. 
    \end{enumerate}
\end{enumerate}

\noindent Next, suppose $\mathcal{C}_3$ corresponds to row~5 of Table~\ref{tab:TypeII}, where $w_4,x_4,y_4,z_4$ are distinct elements from the set $\{c, d, f, h\}$. The array below shows that, for any choice of $w, x, y, z$ and $w_4, x_4, y_4, z_4$, either two of the classes fail to be disjoint or there exists a hyperedge containing one vertex from each of the classes. 
\begin{center} 
\begin{tabular}{|c|c|c|c|c|c|}
\hline
    $xyz$ & $w$ & $x_4y_4z_4$& $w_4$&Hyperedge with $A=abcd$& $\mathcal{C}_2 \cap \mathcal{C}_3$\\\hline
    $cdg$ &$h$& $cdf$&$h$& &$abeh$\\\hline
    $cdg$ &$h$& $cdh$&$f$& &$abef$\\\hline
    $cdg$ &$h$& $cfh$&$d$& $abef, acfh$&\\\hline
    $cdg$ &$h$& $dfh$&$c$& $abef, adfh$&\\\hline
    $cdh$ & $g$ &&& &$abeg$\\\hline
    $cgh$& $d$& &&$acgh, abeg$&\\\hline
    $dgh$& $c$& &&$adgh, abeg$&\\\hline
\end{tabular}
\end{center}
						
\noindent Finally, suppose $\mathcal{C}_3$ corresponds to row~6 of Table~\ref{tab:TypeII}, where $w_5, x_5, y_5, z_5$ are distinct elements from the set $\{b, d, g, h\}$. The following array shows that conclusion of lemma holds in all the cases. 
\begin{center}
\begin{tabular}{|c|c|c|c|c|c|}
\hline
    $xyz$ & $w$ & $x_5y_5z_5$& $w_5$&Hyperedge with $A=abcd$& $\mathcal{C}_2 \cap \mathcal{C}_3$\\\hline
    $cdg$ &$h$& $bdg$&$h$& &$aefh$\\\hline
    $cdg$ &$h$& $bdh$&$g$& $abef, aceg$&\\\hline
    $cdg$ &$h$& $bgh$&$d$& $abeh, adef$&\\\hline
    $cdg$ &$h$& $dgh$&$b$& &$abef$\\\hline
    $cdh$ & $g$ &$bdg$&$h$& $abeg, aceh$&\\\hline
    $cdh$ & $g$ &$bdh$&$g$& &$aefg$\\\hline
    $cdh$ & $g$ &$bgh$&$d$& $abeg, adef$&\\\hline
    $cdh$ & $g$ &$dgh$&$b$& &$abef$\\\hline
    $cgh$& $d$& $bdg$&$h$&$abef, aceh$&\\\hline
    $cgh$& $d$& $bdh$&$g$&$abef, aceg$&\\\hline
    $cgh$& $d$& $bgh$&$d$&&$adef$\\\hline
    $cgh$& $d$& $dgh$&$b$&&$abef$\\\hline
    $dgh$& $c$& &&&$acef$\\\hline
\end{tabular}
\end{center}
\end{proof}      

													
\begin{lemma}
\label{lem3}
Consider a partition of the vertices of $3$-$QI(8,2)$ into classes of size five, each being a maximum strongly independent set. Among any three such classes, if two are of Type~III, then there exist three vertices, one from each class, that together form a hyperedge.
\end{lemma}
\begin{proof}
Let $A = abcd$ be a vertex in a maximum strongly independent set $\mathcal{C}_1$, and let $\mathcal{C}_2$ and $\mathcal{C}_3$ be Type~III maximum strongly independent sets. Since both $\mathcal{C}_2$ and $\mathcal{C}_3$ are maximal, each contains a vertex intersecting $A$ in exactly two elements. Without loss of generality, let $abef \in \mathcal{C}_2$ be such a vertex. The first row of Table~\ref{tab:TypeIII} represents a Type~III class $\mathcal{C}_2$ containing $abef$, where $x$ and $y$ are distinct elements from the set $\{b, e, f\}$. The remaining rows describe all possible forms of $\mathcal{C}_3$, determined by whether its first vertex intersects $abef$ in one, two, or three elements. As in Lemma~\ref{lemma1}, the first entries in rows~2–6 intersect $A$ in exactly two elements.

\begin{table}[htbp]
\centering
\caption{Type III classes $\mathcal{C}_2$ and $\mathcal{C}_3$}
\begin{tabular}{|c|c|c|c|c|c|}
\hline
$\mathcal{C}_2$& $abef$&$acxy$&$adxy$&$agxy$&$ahxy$\\\hline
\multirow{5}{*}{$\mathcal{C}_3$}&$acgh$&$abx_1y_1$&$adx_1y_1$&$aex_1y_1$&$afx_1y_1$\\\hhline{~-----}
&$aceg$&$abx_2y_2$&$adx_2y_2$&$afx_2y_2$&$ahx_2y_2$\\\hhline{~-----} 
&$abgh$&$acx_3y_3$&$adx_3y_3$&$aex_3y_3$&$afx_3y_3$\\\hhline{~-----}
& $abeg$&$acx_4y_4$&  $adx_4y_4$&  $afx_4y_4$& $ahx_4y_4$\\\hhline{~-----}
&$acef$&$abx_5y_5$&$adx_5y_5$&$agx_5y_5$&$ahx_5y_5$\\\hline
\end{tabular}
\label{tab:TypeIII}
\end{table}      
                        
If $\mathcal{C}_3$ corresponds to row~2 of Table~\ref{tab:TypeIII}, with $x_1,y_1 \in\{c, g, h\}$, then the array below shows that for any choices of $x, y$ and $x_1,y_1$, either the classes are not pairwise disjoint or there exists a hyperedge containing one vertex from each of $\mathcal{C}_1,\mathcal{C}_2,$ and $\mathcal{C}_3$.
					
\begin{center}
\begin{tabular}{|c|c|c|c|}
\hline
    $xy$ & $x_1y_1$&Hyperedge with $A=abcd$\\\hline
    $be$ &$cg$& $abef, aceg$\\\hline
    $be$ &$ch$& $abef, aceh$\\\hline
    $be$ &$gh$& $abeg, acgh$\\\hline
    $bf$ &$cg$& $abef, aceg$\\\hline
    $bf$ &$ch$& $abef, aceh$\\\hline
    $bf$ &$gh$& $abfg,acgh$\\\hline
    $ef$&$cg$&$abef,aceg$\\\hline
    $ef$&$ch$&$abef, aceh$\\\hline
    $ef$&$gh$&(separately discussed in Case 1)\\\hline
\end{tabular}
\end{center}
                        
\begin{enumerate}
\item[Case 1.] If $xy=ef$ and $x_1y_1=gh$, then the structure of $\mathcal{C}_1$ given by Theorem~\ref{types} leads to the following conclusions.
    \begin{enumerate}
    \item \textbf{$\mathcal{C}_1$ of Type~I:} In this case, $aefg$ appears in both $\mathcal{C}_1$ and $\mathcal{C}_2$.
    \item \textbf{$\mathcal{C}_1$ of Type~II:} Here, $\mathcal{C}_1$ has the structure:
        \begin{center}
	\begin{tabular}{|c|c|c|c|c|}
        \hline
	$abcd$&$ax'y'z'$,& $abcw'$,& $abdw'$&$acdw'$\\
            & where $x',y',z'\notin \{b,c,d\}$&where $w'\notin  \{b,c,d,x,y,z\}$&&\\
        \hline
        \end{tabular}
	\end{center}
    \smallskip
Then, for any choice of $w',x', y',z'$, two of the classes are not disjoint as shown in the following array. 
        \begin{center}
        \begin{tabular}{|c|c|}
        \hline
        $x'y'z'$  & $\mathcal{C}_1 \cap \mathcal{C}_2$ or $\mathcal{C}_1 \cap \mathcal{C}_3$ \\\hline
        $efg$  & $aefg$\\\hline
        $efh$&$aefh$\\\hline
        $egh$&$aegh$\\\hline
        $fgh$&$afgh$\\\hline
        \end{tabular}
        \end{center}
        \smallskip
    \item \textbf{$\mathcal{C}_1$ of Type III:} In this case, $\mathcal{C}_1$ has the structure:
	\begin{center}
        \begin{tabular}{|c|c|c|c|c|}
        \hline
	$abcd$&$ax'y'e,$& $ax'y'f$&$ax'y'g$&$ax'y'h$\\
              &$x', y'\in\{b, c, d\}$ and $x' \neq y'$&&&\\\hline
	\end{tabular} 
        \end{center}
        \smallskip
    If $x'y' = bc$ or $cd$, then the vertices $ax'y'e$, $aefg$, and $acgh$ form a hyperedge.  Whereas if $x'y' = bd$, then $ax'y'e$, $aefg$, and $adgh$ form a hyperedge. 
    \end{enumerate}
\end{enumerate}	
If $\mathcal{C}_3$ corresponds to row~3 of Table~\ref{tab:TypeIII}, then $abef$ and $aceg$ form a hyperedge with $A$. If $\mathcal{C}_3$ corresponds to row~4 of Table~\ref{tab:TypeIII}, where $x_3, y_3\in\{b, g, h\}$, then the array below shows that for any choice of $x, y$ and $x_3, y_3$ either two of the classes are not disjoint, or there exists a hyperedge containing one vertex from each class.\\
    \begin{center}
    \begin{tabular}{|c|c|c|c|}
    \hline
    $xy$ & $x_3y_3$&Hyperedge with $A=abcd$& $\mathcal{C}_2 \cap \mathcal{C}_3$\\
    \hline
    $be$ &$bg$& &$abeg$\\\hline
    $be$ &$bh$& &$abeh$\\\hline
    $be$ &$gh$& $abeg, acgh$&\\\hline
    $bf$ &$bg$& &$abfg$\\\hline
    $bf$ &$bh$& & $abfh$\\\hline
    $bf$ &$gh$& $abfg,acgh$&\\\hline
    $ef$&$bg$&$acef, abeg$&\\\hline
    $ef$&$bh$&$acef, abeh$&\\\hline
    $ef$&$gh$&(separately discussed in Case 2)&\\\hline
    \end{tabular}
    \end{center}
                            
\begin{enumerate}
\item[Case 2.]
If $xy=ef$ and $x_3y_3=gh$, then the structure of $\mathcal{C}_1$ given by Theorem~\ref{types} leads to the following conclusions.
    \begin{enumerate}
    \item \textbf{$\mathcal{C}_1$ of Type~I:} Here $aefg$ lies in both $\mathcal{C}_1$ and $\mathcal{C}_2$.
    \item \textbf{$\mathcal{C}_1$ of Type II:} In this case, $\mathcal{C}_1$ is 
        \begin{center}
        \begin{tabular}{|c|c|c|c|c|}
         \hline
          $abcd$& $ax'y'z'$,& $abcw'$,& $abdw'$& $acdw'$\\
           &where $x',y',z'\notin \{b,c,d\}$&where $w'\notin \{b,c,d,x,y,z\}$&&\\\hline
         \end{tabular}
         \end{center}
         \smallskip
    and the following array shows that for any choice of $w',x', y',z'$, the classes are not disjoint. 
        \begin{center}
        \begin{tabular}{|c|c|}
        \hline
        $x'y'z'$  & $\mathcal{C}_1 \cap C_2$ or $\mathcal{C}_1 \cap \mathcal{C}_3$ \\\hline
        $efg$  & $aefg$\\\hline
        $efh$&$aefh$\\\hline
        $egh$&$aegh$\\\hline
        $fgh$&$afgh$\\\hline
        \end{tabular}
        \end{center}
        \smallskip
    \item \textbf{$\mathcal{C}_1$ of Type III:} In this case, $\mathcal{C}_1$ is 
	\begin{center}
	\begin{tabular}{|c|c|c|c|c|}
        \hline
	$abcd$&$ax'y'e$,& $ax'y'f$&$ax'y'g$&$ax'y'h$\\
        &$x', y'\in\{b, c, d\}$ and $x' \neq y'$&&&\\\hline
        \end{tabular}
	\end{center}
    \smallskip
    If $x'y'=bc$ or $bd$, then $ax'y'e$, $aefg$ and $abgh$ form a hyperedge, whereas if $x'y'=cd$, then $ax'y'e$, $aefg$ and $acgh$ form a hyperedge.
    \end{enumerate}
\end{enumerate}
Next, suppose $\mathcal{C}_3$ corresponds to row~5 of Table~\ref{tab:TypeIII}, where $x_4, y_4\in\{b, e, g\}$, then the array below shows that for any choice of $x, y$ and $x_4, y_4$ either two of the classes are not disjoint, or there exists a hyperedge containing one vertex from each class.\\ 

\begin{center}
\begin{tabular}{|c|c|c|c|}
\hline
$xy$ & $x_4y_4$&Hyperedge with $A=abcd$& $\mathcal{C}_2 \cap \mathcal{C}_3$\\\hline
$be$ &$be$& &$abef$\\\hline
$be$ &$bg$& &$abeg$\\\hline
$be$ &$eg$& &$abeg$\\\hline
$bf$ &$be$& &$abef$\\\hline
$bf$ &$bg$& & $abfg$\\\hline
$bf$ &$eg$& $abef,aceg$&\\\hline
$ef$&$be$&&$abef$\\\hline
$ef$&$bg$&$acef, abeg$&\\\hline
$ef$&$eg$&&$aefg$\\\hline
\end{tabular}
\end{center}

Finally, suppose $\mathcal{C}_3$ corresponds to row~6 of Table~\ref{tab:TypeIII}, where $x_5, y_5\in\{c, e, f\}$, then the array below shows that the conclusion of the lemma holds whenever classes are disjoint.\\
  
\begin{center}
\begin{tabular}{|c|c|c|c|}
\hline
$xy$ & $x_5y_5$& Hyperedge with $A=abcd$& $\mathcal{C}_2 \cap \mathcal{C}_3$\\
\hline
$ef$&&&$acef$\\\hline
&$ef$&&$abef$\\\hline
$be$ &$ce$& &$abce$\\\hline
$be$ &$cf$& $abef, acfg$&\\\hline
$bf$ &$ce$& $abef, aceg$&\\\hline
$bf$ &$cf$& & $abcf$\\\hline
\end{tabular}
\end{center}
\end{proof}

\begin{lemma}
\label{lem4}
Consider a partition of the vertices of $3$-$QI(8,2)$ into classes of size five, each being a maximum strongly independent set. Among any three such classes, if all three are of different types, then there exist three vertices, one from each class, that together form a hyperedge.
\end{lemma}
\begin{proof}
Let $\mathcal{C}_1, \mathcal{C}_2,$ and $\mathcal{C}_3$ be three classes of Type I, II, and III, respectively. Let $A=abcd\in \mathcal{C}_3$. Since $\mathcal{C}_1$ and $\mathcal{C}_2$ are maximum strongly independent sets, they contain at least one vertex that intersects $A$ in exactly two elements. Without loss of generality, let $abef\in\mathcal{C}_2$ be one such vertex. The first row of Table~\ref{tab:Type12} represents a Type~III class $\mathcal{C}_3$ containing the vertex $A=abcd$; the second row represents a Type~II class $\mathcal{C}_2$ containing the vertex $abef$, where $w, x, y,$ and $z$ are distinct elements from $\{c, d, g, h\}$. Each of the remaining rows corresponds to a possible form of Type~I class $\mathcal{C}_1$, depending on whether its first vertex intersects $abef$ in one, two, or three elements. As in Lemma~\ref{lemma1}, the first vertex in rows 3-7 intersects $A$ in exactly two elements. \\

\begin{table}[htbp]
    \centering
    \caption{Type I class $\mathcal{C}_1$, Type II class $\mathcal{C}_2$, and Type III class $\mathcal{C}_3$}
    \begin{tabular}{|c|c|c|c|c|c|}
    \hline
    $\mathcal{C}_3$&$abcd$& $ax'y'e$& $ax'y'f$&$ax'y'g$&$ax'y'h$\\
                  &&$x', y'\in\{b,c,d\}$ and $x,\neq y'$&&&\\\hline
    $\mathcal{C}_2$&$abef$&$axyz$&$abew$&$abfw$&$aefw$ \\\hline
    \multirow{5}{*}{$\mathcal{C}_1$}
    &$acgh$&$abde$&$abdf$&$abef$&$adef$\\\hhline{~-----}
    &$aceg$&$abdf$&$abdh$&$abfh$&$adfh$\\\hhline{~-----}
    &$abgh$&$acde$&$acdf$&$adef$&$acef$\\\hhline{~-----}
    &$abeg$&$acdf$&$acdh$&$adfh$&$acfh$\\\hhline{~-----}
    &$acef$&$abdg$&$abdh$& $adgh$&$abgh$\\\hline     
    \end{tabular}
    \label{tab:Type12}
\end{table}

Since row~2 and row~3 in Table~\ref{tab:Type12} share a vertex $abef$, row~3 cannot serve as $\mathcal{C}_1$. If $\mathcal{C}_1$ corresponds to row~4, then $abef$ and $aceg$ form a hyperedge with $A=abcd$. Similarly, if $\mathcal{C}_1$ corresponds to row~5 of Table~\ref{tab:Type12}, then the array below shows that for any choice of $w, x, y,$ and $z$ either $\mathcal{C}_1$ and $\mathcal{C}_2$ are not disjoint or there exists a hyperedge containing one vertex from each class. 

\begin{center}
\begin{tabular}{|c|c|c|}
\hline
$xyz$ & Hyperedge with $A=abcd$& $\mathcal{C}_1 \cap \mathcal{C}_2$ \\\hline
$cdg$  & $abfh, adef$&\\\hline
$cdh$&$abfg, adef$&\\\hline
$cgh$&&$adef$\\\hline
$dgh$&&$acef$\\\hline
\end{tabular}
\end{center}

Next, if $\mathcal{C}_1$ corresponds to row~6 of Table~\ref{tab:Type12}, then $A=abcd$, $abef$, and $adfh$ form a hyperedge. Finally, if $\mathcal{C}_1$ corresponds to row~7 of Table~\ref{tab:Type12}, then the following array shows that for any choice of $w, x, y,$ and $z$ either the classes fail to be disjoint or there exists a hyperedge containing one vertex from each class. 
    
\begin{center}               
\begin{tabular}{|c|c|c|c|}
\hline
$xyz$& $x'y'$  &Hyperedge & $\mathcal{C}_2 \cap \mathcal{C}_3$ or $\mathcal{C}_1 \cap \mathcal{C}_2$ \\\hline
$cdg$ & & $abcd, abeh, acef$&\\\hline
$cdh$&& $abcd, abeg, acef$&\\\hline
$cgh$&$bc$ & $abcd, acgh, abdg$& \\\hline
$cgh$&$bd$ & & $abde$ \\\hline
$cgh$&$cd$ & $acde, abef, abdg$& \\\hline
$dgh$&&&$adgh$\\\hline
\end{tabular}
\end{center}
\end{proof}

The preceding lemmas together yield the following theorem, which is central to establishing our subsequent result on the core of $3\text{-}QI(8,2)$.

\begin{theorem}
\label{main}
Let the vertices of $3$-$QI(8,2)$ be partitioned into classes of size five, each a maximum strongly independent set. Then any three such classes contain three vertices—one from each class—that together form a hyperedge.
\end{theorem}

\begin{proof}
Consider a partition of the vertices of $3$-$QI(8,2)$ into classes of size five, each a maximum strongly independent set. By Theorem~\ref{types}, every such class is of Type~I, Type~II, or Type~III. For any choice of three classes, either all three are of different types or at least two are of the same type. If all three classes are of distinct types, the claim follows from Lemma~\ref{lem4}. If at least two classes are of the same type, the conclusion follows from Lemmas~\ref{lemma1}, \ref{lem2}, and~\ref{lem3}. Hence, in every triple of classes, there exist three vertices—one from each class—that form a hyperedge.
\end{proof}

\begin{theorem}
The hypergraph $3\text{-}QI(8,2)$ is a core.
\end{theorem}
\begin{proof}
Let $H=(V,E)$ denote the core of $3\text{-}QI(8,2)$. By Lemma~\ref{core} together with Proposition~\ref{vertextransitive}, $|V|$ must divide the order of $3\text{-}QI(8,2)$, namely $35$. Furthermore, by Proposition~\ref{corechrom} and Proposition~\ref{chrom}, we have $\chi_S(H)=\chi_S(3\text{-}QI(8,2))=7$, which implies $|V|\ge 7$. Hence $|V|\in\{7,35\}$.\\

If $|V|=7$, then hypergraph $H$ cannot be isomorphic to $K_7^3$, because from Lemma~\ref{coreclique} and Proposition~\ref{clique}, $\omega_3(H)=\omega_3(3\text{-}QI(8,2))=4$. Thus, $H$ contains a triple of vertices, say $A,B,C$, that do not form a hyperedge. Since there exists a homomorphism $f\colon 3\text{-}QI(8,2)\to H$, the preimage of each vertex of $H$ under $f$ is a strongly independent set and these seven preimages form a partition of the vertices of $3$-$QI(8,2)$. By Corollary~\ref{ind}, every strongly independent set of $3\text{-}QI(8,2)$ has maximum cardinality $5$. Thus, each preimage must have cardinality exactly $5$, and therefore each is a maximum strongly independent set. By Theorem~\ref{main}, the three preimage classes corresponding to $A, B, C$ contain a hyperedge of $3\text{-}QI(8,2)$, even though $A,B, C$ do not form a hyperedge in $H$. This contradicts the assumption that $f$ is a homomorphism. Hence, $|V|$ cannot be $7$. Therefore, $|V|=35$, and by Lemma~\ref{induced} the core is an induced subhypergraph. Consequently, $H$ is isomorphic to $3\text{-}QI(8,2)$, and thus $3\text{-}QI(8,2)$ is a core.
\end{proof}

\section{Conclusion}
\label{sec:6}
We examined the structural properties of $3\text{-}QI(8,2)$ and proved that it is a core hypergraph. We also established a structural link between qualitative independence hypergraphs and merged Johnson graphs, providing a framework through which properties of one family may be derived from the other. This approach further yields sufficient conditions under which $3\text{-}UQI(n,2)$ and $3\text{-}AUQI(n,2)$ are cores. Determining the core of $3\text{-}QI(n,2)$ for $n>8$ remains an open problem. A natural direction for future work, suggested by Corollary~\ref{final2}, is to analyze the endomorphism monoid of the merged Johnson graph
$
J\!\left(n, \left\lfloor \tfrac{n-1}{2} \right\rfloor \right)_I$, where $I=\{2,3,\ldots, \left\lfloor \tfrac{n-1}{2} \right\rfloor -2\}.
$
Showing that every endomorphism of this graph is an automorphism would settle the core question for $3\text{-}AUQI(n,2)$ for all odd $n\ge 9$.

\subsection*{Acknowledgements}

The first author is grateful to the Department of Science and Technology (DST) for financial support under the DST-INSPIRE Senior Research Fellow scheme with sanction no. DST/INSPIRE Fellowship/2021/IF210377.

\bibliographystyle{ieeetr}

\bibliography{References}

@book{royle2001algebraic,
  title={Algebraic graph theory},
  author={Royle, Gordon F and Godsil, Chris},
  volume={207},
  year={2001},
  publisher={Springer},
address={New York}
}

@misc{micciancio2004using,
  title={Using hypergraph homomorphisms to guess three secrets},
  author={Micciancio, Daniele and Segerlind, Nathan},
  year={2004},
  publisher={manuscript}
}

@article{appl1994compressing,
  title={Compressing inconsistent data},
  author={Korner, J and Lucertini, Mario},
  journal={IEEE Transactions on Information Theory},
  volume={40},
  number={3},
  pages={706--715},
  year={1994},
  publisher={IEEE},
 doi={10.1109/18.335882}
}

@article{meagher2005covering,
  title={Covering arrays on graphs},
  author={Meagher, Karen and Stevens, Brett},
  journal={Journal of Combinatorial Theory, Series B},
  volume={95},
  number={1},
  pages={134--151},
  year={2005},
  publisher={Elsevier},
doi={https://doi.org/10.1016/j.jctb.2005.03.005}
}

@phdthesis{raaphorst2013thesisvariable,
  title={Variable strength covering arrays},
  author={Raaphorst, Sebastian},
  year={2013},
  school={University of Ottawa (Canada)},
doi={https://doi.org/10.1002/jcd.21602}
}

@article{raaphorst2018variable,
  title={Variable strength covering arrays},
  author={Raaphorst, Sebastian and Moura, Lucia and Stevens, Brett},
  journal={Journal of Combinatorial Designs},
  volume={26},
  number={9},
  pages={417--438},
  year={2018},
  publisher={Wiley Online Library},
doi={https://doi.org/10.1002/jcd.21602}
}

@article{fellner1982minimal,
  title={On minimal graphs},
  author={Fellner, Wolf-Dietrich},
  journal={Theoretical Computer Science},
  volume={17},
  number={1},
  pages={103--110},
  year={1982},
  publisher={Elsevier}
}

@article{hell1992core,
  title={The core of a graph},
  author={Hell, Pavol and Ne{\v{s}}et{\v{r}}il, Jaroslav},
  journal={Discrete Mathematics},
  volume={109},
  number={1-3},
  pages={117--126},
  year={1992},
  publisher={Elsevier}
}

@article{nevsetvril1978classescore,
  title={On classes of relations and graphs determined by subobjects and factorobjects},
  author={Ne{\v{s}}et{\v{r}}il, Jaroslav and Pultr, Ale{\v{s}}},
  journal={Discrete Mathematics},
  volume={22},
  number={3},
  pages={287--300},
  year={1978},
  publisher={Elsevier}
}

@article{bauslaugh1995core,
  title={Core-like properties of infinite graphs and structures},
  author={Bauslaugh, Benjamin},
  journal={Discrete mathematics},
  volume={138},
  number={1-3},
  pages={101--111},
  year={1995},
  publisher={Elsevier}
}

@inproceedings{applstevens1998efficient,
  title={Efficient software testing protocols},
  author={Stevens, Brett and Mendelsohn, Eric},
  booktitle={Proceedings of the 1998 conference of the Centre for Advanced Studies on Collaborative research},
  pages={22},
  year={1998}
}

@article{jones2005automorphisms,
  title={Automorphisms and regular embeddings of merged Johnson graphs},
  author={Jones, Gareth A},
  journal={European Journal of Combinatorics},
  volume={26},
  number={3-4},
  pages={417--435},
  year={2005},
  publisher={Elsevier}
}

@article{survey2019methods,
  title={Methods to construct uniform covering arrays},
  author={Torres-Jimenez, Jose and Izquierdo-Marquez, Idelfonso and Avila-George, Himer},
  journal={IEEE Access},
  volume={7},
  pages={42774--42797},
  year={2019},
  publisher={IEEE},
doi={10.1109/ACCESS.2019.2907057}
}

@article{godsil2011cores,
  title={Cores of geometric graphs},
  author={Godsil, Chris and Royle, Gordon F},
  journal={Annals of Combinatorics},
  volume={15},
  pages={267--276},
  year={2011},
  publisher={Springer}
}

@article{godsilnewman33,
author = {Godsil, C. D. and Newman, M. W.},
title = {Independent Sets In Association Schemes},
year = {2006},
issue_date = {August 2006},
publisher = {Springer-Verlag},
address = {Berlin, Heidelberg},
volume = {26},
number = {4},
issn = {0209-9683},
url = {https://doi.org/10.1007/s00493-006-0024-z},
doi = {10.1007/s00493-006-0024-z},
journal = {Combinatorica},
month = aug,
pages = {431–443},
numpages = {13}
}

@inproceedings{alanwilliams1996practical,
  title={A practical strategy for testing pair-wise coverage of network interfaces},
  author={Williams, Alan W and Probert, Robert L},
  booktitle={Proceedings of ISSRE'96: 7th International Symposium on Software Reliability Engineering},
  pages={246--254},
  year={1996},
  organization={IEEE},
    doi = {10.1109/ISSRE.1996.558835}
}

@article {raina,
author = {Thomas, Raina Mary and Akhtar, Yasmeen},
title = {On the Covering Array Number of $3$-uniform Qualitative Independence Hypergraphs},
journal = {Communications in Combinatorics and Optimization},
volume = {},
number = {},
pages = {},
year  = {2026},
publisher = {Azarbaijan Shahid Madani University},
issn = {2538-2128}, 
eissn = {2538-2136}, 
doi = {10.22049/cco.2026.30424.2472},	
url = {https://comb-opt.azaruniv.ac.ir/article_15123.html},

}

@book{berge1984hypergraphs,
  title={Hypergraphs: combinatorics of finite sets},
  author={Berge, Claude},
  volume={45},
  year={1984},
  publisher={Elsevier}
}

@inproceedings{fagin2003data,
  title={Data exchange: getting to the core},
  author={Fagin, Ronald and Kolaitis, Phokion G and Popa, Lucian},
  booktitle={Proceedings of the twenty-second ACM SIGMOD-SIGACT-SIGART symposium on Principles of database systems},
  pages={90--101},
  year={2003}
}

\end{document}